\documentclass[reqno]{amsart}
\usepackage{latexsym}
\usepackage{euscript}
\usepackage{amssymb}
\usepackage{array}      
\usepackage{enumerate}   

\newtheorem{Thm}[equation]{Theorem}
\newtheorem{Lem}[equation]{Lemma}
\newtheorem{Cor}[equation]{Corollary}
\newtheorem{Prop}[equation]{Proposition}

\theoremstyle{remark}
\newtheorem{Notation}[equation]{\bf Notation}
\newtheorem{Rem}[equation]{Remark}

\newtheorem{Def}[equation]{Definition}

\numberwithin{equation}{section}

\newcommand{\R}{\text{\bf R}}           
\newcommand{\C}{\text{\bf C}}           
\newcommand{\N}{\text{\bf N}}           

\newcommand{\Ad}{\text{Ad}}

\newcommand{\rank}{\text{rank}}
\newcommand{\spa}{\text{span}}

\newcommand{\Stab}{\text{Stab}}
\newcommand{\gr}{\text{gr}}

\newcommand{\f}[1]{{\mathfrak{#1}}}           
\newcommand{\fb}{{\mathfrak b}}
\newcommand{\fg}{{\mathfrak g}}
\newcommand{\fh}{{\mathfrak h}}
\newcommand{\fk}{{\mathfrak k}}
\newcommand{\fl}{{\mathfrak l}}

\newcommand{\fn}{{\mathfrak n}}
\newcommand{\fnbar}{{\overline{\mathfrak n}}}
\newcommand{\fp}{{\mathfrak p}}
\newcommand{\fq}{{\mathfrak q}}

\newcommand{\fu}{{\mathfrak u}}

\newcommand{\fB}{{\mathfrak B}}
\newcommand{\fF}{{\mathfrak F}}

\newcommand{\ga}{\alpha}
\newcommand{\gb}{\beta}

\newcommand{\gl}{\lambda}
\newcommand{\gL}{\Lambda}

\newcommand{\gD}{\Delta}
\renewcommand{\ge}{\epsilon}
\renewcommand{\ge}{\epsilon}
\newcommand{\tge}{\tilde{\epsilon}}

\newcommand{\gs}{\sigma}

\newcommand{\gt}{\theta}

\newcommand{\Cal}{\mathcal}
\newcommand{\Eul}{\EuScript}
\newcommand{\fQ}{{\Eul Q}}

\renewcommand{\bar}[1]{\overline{#1}}
\newcommand{\IC}{{\bf{I^{^{\bullet}}\!C}}}

\newcommand{\IP}[2]{\langle#1\,, #2\rangle}     

\newcommand{\Cl}{\Cal C\hskip-1pt\ell}
\newcommand{\sym}{\mathrm{sym^\dagger}}
\newcommand{\Cg}{\Eul C_g}
\newcommand{\Chc}{\Eul C_{HC}}
\newcommand{\supp}{\text{supp}}
\newcommand{\ann}{\text{ann}}
\newcommand{\fX}{{\Eul X}}
\usepackage{mathdots}  
\newcommand{\cW}{{\Cal W}}

\begin{document}

\parskip=3pt 

\title[]{Characteristic cycles of highest weight Harish-Chandra modules for $Sp(2n,\R)$}
\author{L. Barchini}
\author{R. Zierau}

\address{Oklahoma State University \\
   Mathematics Department \\
   Stillwater, Oklahoma 74078}
      \email{leticia@math.okstate.edu}
      \email{zierau@math.okstate.edu}
 
\thanks{Research supported by  NSA grant H98230-13-1-0268}
\subjclass[2010]{22E46}   
\keywords{Harish-Chandra module, highest weight module, symplectic group, characteristic cycle}         


\maketitle



\begin{abstract} Characteristic cycles, leading term cycles, associated varieties and Harish-Chandra cells are computed for the family of highest weight Harish-Chandra modules  for $Sp(2n,\R)$ having regular integral infinitesimal character.
\end{abstract}

\section*{Introduction}

Two important invariants of Harish-Chandra modules are the associated cycle and the characteristic cycle.  For example, a conjecture of Vogan and its proof by Schmid and Vilonen  (\cite{SchmidVilonen00}) provides deep connections between the algebraically defined associated cycle and the global character of an admissible representation.  However, there is no known method for computing either invariant in any generality.  In this article characteristic cycles are computed for the family of highest weight Harish-Chandra modules having regular integral infinitesimal character for $Sp(2n,\R)$.  

Associated cycles for the unitary highest weight Harish-Chandra modules are known for the classical groups (\cite{NishiyamaOchiaiTaniguchi01}).  The characteristic cycles for the unitary highest weight Harish-Chandra modules of regular integral infinitesimal character are all just the conormal bundle of the support, since they are cohomologically induced from a one dimensional representation, so the support has smooth closure.   It is well-known that characteristic varieties for (nonunitary) highest weight Harish-Chandra modules need not be irreducible; low dimensional examples are easy to find.  In this article we determine the characteristic cycles of all highest weight Harish-Chandra modules with regular integral infinitesimal character for $Sp(2n,\R)$; the statement is contained in Theorem \ref{thm:main}.  In addition, associated varieties and leading term cycles are computed, and the Harish-Chandra cells are described.   The statements are given in terms of clans, which parametrize the $K$-orbits in the flag variety, a subset of which parametrize the highest weight Harish-Chandra modules.

The method used to understand the characteristic cycles, Harish-Chandra cells, etc., is inductive in nature.  Each clan for a highest weight Harish-Chandra module (of infinitesimal character $\rho$) for $Sp(2n,\R)$ is obtained easily from such clans for $Sp(2(n-1),\R)$.  Characteristic cycles, leading term cycles and Harish-Chandra cells (and all other information) is given in terms of the same information for $Sp(2(n-1),\R)$.

There is some overlap between this paper and the results of \cite{BoeFu97}.  The intersection homology sheaves for $B$-orbit closures in the generalized flag variety of lagrangian planes is considered in \cite{BoeFu97}.  In that article an algorithm to compute characteristic cycles of these intersection homology sheaves is given.  Although it is not immediate, this algorithm can be related to the computation of characteristic cycles of highest weight Harish-Chandra modules considered in the present article. We do not make this connection, but use very different methods.

The organization of this paper is as follows.  The first section gives some well-known generalities on characteristic cycles and highest weight Harish-Chandra modules.  In section two we review the clan notation, determine the clans for the highest weight Harish-Chandra modules and express the $T_{\ga\gb}$ operators in the terms of clans.  Section 3 determines the Harish-Chandra cells.  The fourth section contains a key lemma relating the characteristic cycles corresponding to a particular clan for $Sp(2n,\R)$ to characteristic cycles for $Sp(2(n-1),\R)$.  The computation of all characteristic cycles is given in Section 5.

\noindent\emph{Acknowledgement}  The authors have used the ATLAS software as a valuable tool for computing examples.  Although this paper is independent of the software, computations using the ATLAS have played a significant role in this project.


\section{Preliminaries}
\subsection{The symplectic group}\label{subsec:sp} For this article we consider the pair of complex groups  $(G,K)=(Sp(2n),GL(n))$.  The real group in the background is $Sp(2n,\R)$.  We use the $n\times n$ matrix
\begin{equation*}\label{eqn:S}
    S_n:= \begin{pmatrix}  && 1  \\ &
    \iddots &  \\1&&
     \end{pmatrix},
\end{equation*}
to define
\begin{equation}\label{eqn:spf}
J:=\begin{pmatrix}
   0_n & S_n \\  -S_n & 0_n \end{pmatrix}.
\end{equation}  
The realization of the complex group $G=Sp(2n)$ that we use is
$$G:=\{g\in M_n(\C) : g^tJg=J\}.
$$
The matrix $S_n$ gives an `antidiagonal transpose' $^\dagger\!X=S_nX^tS_n$;  we let $\sym(n)=\{A\in M_n:A=^\dagger\!\!A\}$.  
The Lie algebra of $G$ is then
\begin{equation*}
\fg=\left\{\begin{pmatrix}
   A & B \\  C & -^\dagger\!A \end{pmatrix}  : B,C\in \sym(n)\right\}.
\end{equation*}
The subgroup $K$ is the fixed point set of the involution $\gt=\Ad(I_{n,n})$, with 
\begin{equation*}I_{n,n}:=\begin{pmatrix}
   I_n & 0 \\  0 & -I_n \end{pmatrix}.
\end{equation*}  
Therefore,
\begin{equation*}
K=\left\{\begin{pmatrix}
   a& 0 \\  0 & ^\dagger a^{-1} \end{pmatrix} : a\in GL(n)\right\} \text{ and }\fk=\left\{\begin{pmatrix}
   A & 0 \\  0 & -^\dagger\!A \end{pmatrix}\right\}.
\end{equation*}  

The orthogonal complement of $\fk$ (with respect to the Killing form) is 
$$\fp:=\left\{\begin{pmatrix}
   0 & B \\  C & 0 \end{pmatrix}  : B,C\in \sym(n)\right\}.
$$  
This gives the decomposition 
$\fg=\fk+\fp$; $\fp$ decomposes into the direct sum of two irreducible $K$-subrepresentations,
\begin{equation}\label{eqn:ht}
\fp=\fp_+\oplus\fp_{-}
\end{equation}
with
\begin{equation*}
\fp_+=\left\{\begin{pmatrix}
   0 & B \\  0 & 0  
\end{pmatrix}:B\in \sym(n)\right\}
\end{equation*}
and $\fp_-$ the transpose of $\fp_+$.

The diagonal matrices in $\fg$, i.e., those  matrices of the form $H_A:=\begin{pmatrix} A & 0 \\ 0 & -^\dagger A\end{pmatrix}\in \fk$ with 
$$A=\begin{pmatrix}
   t_1 & &  \\\ &\ddots & \\ && t_n \end{pmatrix},
$$
form a  Cartan subalgebra $\fh$ of both $\fk$ and $\fg$.  Let $\ge_j\in\fh^*$ be defined by $\ge_j(H_A)=t_j$.  Then the roots are $\{\ge_i-\ge_j:1\leq i,j\leq n, i\neq j\}\cup\{\pm(\ge_i+\ge_j):1\leq i\leq j\leq n\}$.  We once and for all fix the positive system
\begin{equation}\label{eqn:ps}
\gD^+=\gD^+(\fh,\fg)=\{\ge_i-\ge_j:1\leq i<j\leq n\}\cup\{\ge_i+\ge_j:1\leq i\leq j\leq n\}.
\end{equation}
The simple roots are 
\begin{equation}\label{eqn:simple}
S:=\{\ga_j=\ge_j-\ge_{j+1}:j=1,2,\dots,n-1\}\cup\{\ga_n=2\ge_n\}.
\end{equation}
The set of roots in $\fp_+$ is $\gD(\fp_+)=\{\ge_i+\ge_j:1\leq i\leq j\leq n\}$.

We fix a Borel subalgebra $\fb=\fh+\sum_{\ga\in\gD^+}\fg^{(\ga)}$.  This is the Lie algebra of upper triangular matrices in $\fg$.  The connected subgroup of $G$ with Lie algebra $\fb$ is a Borel subgroup denoted by $B$.

The Weyl group $W$ consists of all permutations and sign changes of $\{\ge_i\}$.  This may be expressed in several ways.  We will usually write elements of $W$ as 
\begin{equation}\label{eqn:w-short}
w=(w_1w_2\dots w_n), \text{ when } w(\ge_j)=
\begin{cases}\ge_{w_j}, &\text{if }w_j>0  \\
              -\ge_{-w_j},  &\text{if }w_j<0.
\end{cases}
\end{equation}

Many of our arguments will be inductive in nature,  reducing to the smaller pair $(G',K')=(Sp(2(n-1),GL(n-1))$.  The group $G'$ is embedded in $G$ so that a Cartan subalgebra is $\fh'=\{H\in \fh : \ge_1(H)=0\}$ and $\fg'=\fh'+\sum_{\ga\in\gD(\fh,\fg),\IP{\ga}{\ge_1}=0}\fg^{(\ga)}$.  We use the notation $K'$ for $K\cap G'$ and $B'$ for $B\cap G'$.  Similar notation is used for other subgroups of $G'$ and various subalgebras of its Lie algebra $\fg'$.


\subsection{Characteristic cycles}\label{subsec:1-cc}  Let $\Cal M(\fg,K)$ be the category of finitely generated $(\fg,K)$- modules of infinitesimal character $\rho$ and let $\Cal M_c(\Cal D_\fB,K)$ be the category of coherent $K$-equivariant $\Cal D_\fB$-modules on the flag variety $\fB$. Localization gives an equivalence of these two categories (\cite{BeilinsonBernstein81}).  The definition of the characteristic cycle of a $\Cal D$-module, along with the first properties, is contained in \cite{BorhoBrylinski85}.  Included there are the following facts.  Let $\fX\in\Cal M_c(\Cal D_\fB,K)$.
\begin{enumerate}[(a)]
\item The characteristic cycle of $\fX$ is of the form
$$CC(\fX)=\sum_{\fQ\in K\backslash\fB} m_\fQ[\bar{T_\fQ^*\fB}],
$$
viewed as an element of top degree Borel-Moore homology of the conormal variety $\displaystyle{\cup_{\fQ\in K\backslash\fB} \bar{T_\fQ^*\fB}}$.  The $m_\fQ$'s are nonnegative integers called the \emph{multiplicities}.
\item In the above formula, if $m_\fQ\neq 0$, then $\fQ\subset\supp(\fX)$.  In fact $\fQ$ is in the singular locus of $\supp(\fX)$.
\item If $\fX$ is irreducible, the support of $\fX$ is the closure of a single $K$-orbit $\fQ\subset\fB$ and $m_\fQ=1$.
\item If $\supp(\fX)=\bar{\fQ}$ is a smooth subvariety of $\fB$, then $CC(\fX)=[\bar{T_\fQ^*\fB}]$.
\end{enumerate}

\begin{Notation}  Given $X\in\Cal M(\fg,K)$ we write $\fX$ for its localization, a $\Cal D$-module on $\fB$.  We refer to the support of $\fX$ as the support of $X$.  We also write $CC(X)$ for the characteristic cycle of the localization $\fX$ of $X$.  A Harish-Chandra module written as $X_\fQ$ is assumed to have support $\bar{\fQ}$.  
\end{Notation}

The associated variety of a Harish-Chandra module is defined in \cite{Vogan91}.  It is a union of $K$-orbits in $\fp$.  The characteristic cycle and associated variety are related through the moment map $\mu:T^*\fB\to\Cal N$.  If $CC(X)=\sum m_\fQ[\bar{T_\fQ^*\fB}]$, then
\begin{equation*}
  AV(X)=\bigcup_{m_\fQ\neq 0} \mu(\bar{T_\fQ^*\fB}).
\end{equation*}
The \emph{leading term cycle} is defined to be
$$
LTC(X)=\sum m_\fQ[\bar{T_\fQ^*\fB}],$$
summing over all $\fQ$ with $\dim(\mu(\bar{T_\fQ^*\fB})) =\dim(AV(X))$.

In the category $\Cal M(\fg,B)$ of finitely generated $(\fg,B)$-modules of infinitesimal character $\rho$, the picture is entirely similar.  The supports of the localizations are the closures of $B$-orbits in $\fB$, the Schubert varieties.  We denote the $B$-orbit of $w\cdot\fb$ by $B_w$; the Schubert variety $\bar{B_w}$ is denoted by $Z_w$.  The moment map image $\mu(\bar{T_{B_w}^*\fB})$ is the orbital variety $\bar{B\cdot\fn\cap\fn^w}$, $\fn^w:=\Ad(w)\fn$.  There is an equivalence of categories between $\Cal M(\fg,B)$ and the category $\Cal M_c(\Cal D_\fB,B)$ of coherent $B$-equivariant $\Cal D$-modules on $\fB$.  Facts (a)-(d) hold in this setting, along with the additional fact that if $\bar{T_{B_y}^*\fB}$ occurs in $CC(L_w)$, then the $\tau$ invariant of $y$ contains the $\tau$-invariant of $w$. The associated variety is a union of orbital varieties.

It will be important for us to to view highest weight Harish-Chandra modules as lying in \emph{both} categories $\Cal M(\fg,K)$ and $\Cal M(\fg,B)$.  The characteristic cycle of a $\Cal D_\fB$-module is defined independent of which category we are in.  This situation is discussed further in \S\ref{subsec:bo} and \S\ref{sec:lemma}.

   
\subsection{Highest weight Harish-Chandra modules}\label{subsec:hwt-hc} The group $G_\R$ has infinite dimensional irreducible representations having Harish-Chandra modules with highest weight vectors.  In general a connected, simple Lie group has such representations when $G_\R$ is of hermitian type.  

Let $\fp=\fp_+\oplus\fp_-$ be as in (\ref{eqn:ht}).  In (\ref{eqn:ps}) we have fixed a positive system of roots $\gD^+=\gD^+(\fh,\fg)$ which contains $\gD(\fp_+)$.  If $X$ is a highest weight Harish-Chandra module, then $X$ contains a vector $v^+$ annihilated by all root vectors $X_\gb$, $\gb\in\gD(\fp_+)$ or annihilated by all $X_\gb$, $\gb\in\gD(\fp_-)$.   We consider highest weight Harish-Chandra modules with respect to $\gD^+$, i.e., those with weight vectors annihilated by $X_\gb$, $\gb\in\gD(\fp_+)$.

We are concerned with Harish-Chandra modules of infinitesimal character 
\begin{equation*}
  \rho=\frac12\sum_{\ga\in\gD^+}\ga=(n,n-1,\dots,2,1)
\end{equation*}
The irreducible highest weight modules of infinitesimal character $\rho$ are parameterized by  the Weyl group $W$; we write  $L_{-w\rho}$ for the irreducible quotient of 
$$
\Cal U(\fg)\underset{\Cal U(\fb)}{\otimes}\C_{-w\rho-\rho}.
$$
For a highest weight Harish-Chandra module $v^+$ is annihilated by $X_\gb\in\gD_c^+\subset\gD^+$ and is $K$-finite, so its weight $-w\rho-\rho$ is $\gD_c^+$-dominant.  When $-w\rho-\rho$ is $\gD_c^+$-dominant, $L_{-w\rho}$ is the irreducible quotient of 
$$
\Cal U(\fg)\underset{\Cal U(\fk+\fp_+)}{\otimes}E_{-w\rho-\rho},
$$
where $E_{-w\rho-\rho}$ is the irreducible finite dimensional representation of $\fk$ of highest weight $-w\rho-\rho$.  We conclude that $L_{-w\rho}$ is a Harish-Chandra module exactly when $-w\rho$ is $\gD_c^+$-dominant.

\begin{Def}\label{def:our-w}
$\cW:=\{w\in W : -w\rho\text{ is $\gD_c^+$-dominant}\}.$
\end{Def}
\noindent It follows that $^\#\cW=^\#\!\!\left(W/W_c\right)=2^n$.  Note that 
\begin{align*}
&\text{if } w=(\text{-}1\,\text{-}2\,\dots\,\text{-}n), \text{ then } L_{-w\rho}=\C,\text{ the trivial representation, and}  \\
&\text{if }w=(n\,\dots \, 2\,1),\text{ then } L_{-w\rho}\text{ is in the holomorphic discrete series}.
\end{align*}
$\cW$ may be described as those $w=(w_1\dots w_n)\in W$ with entries $-1,-2,\dots,-k$ appearing in decreasing order (from left to right) and $n,n-1,\dots,n-k+1$ also appearing in decreasing order, for some $k=0,1,\dots,n$.  For example, when $n=3$, $$\cW=\left\{(321),(\text{-}132),(3\text{-}12),(32\text{-}1),(\text{-}1\text{-}23),(\text{-}13\text{-}2),(3\text{-}1\text{-}2),(\text{-}1\text{-}2\text{-}3)\right\}.$$

It follows from the definition that the associated variety of a highest weight Harish-Chandra module is contained in $\fp_+$; see \cite{NishiyamaOchiaiTaniguchi01}.  The $K$-orbits in $\fp_+$ have a particularly nice form.  They are 
\begin{equation}\label{eqn:Korbit}
\Cal O_k=\left\{\begin{pmatrix} 0 & X \\ 0 & 0 \end{pmatrix}; X\in\sym(n), \rank(X)=k\right\}, k=0,1,\dots n.
\end{equation}
Since $\Cal O_k\subset\bar{\Cal O}_{k+1}$, the associated variety of any irreducible highest weight Harish-Chandra module is the closure of exactly one $\Cal O_k$.


\section{$K$-orbits and Schubert varieties}
\subsection{Clans}\label{subsec:clans} 
 Many of the results and arguments of this article are expressed in terms of clans.  Clans give a parametrization of $K$-orbits in the flag variety for a given classical group.  In this section we review the clan notation and some basic facts that we will need.  In \cite{MatsukiOshima88}, $K$-orbits in $\fB$ are classified by signed involutions. We follow the description of clans  given in \cite{Yamamoto97}, which is in terms of flags.

For the pair $(G,K)=(Sp(2n),GL(n))$ the clans are $2n$-tuples $c=(c_1,\dots,c_{2n})$ satisfying the following.\\
\indent (a) Each $c_i$ is $+,-$ or a natural number.  \\
\indent (b) If $c_i\in\N$, then $c_j=c_i$ for exactly one $j\neq i$.  \\
\indent (c) The number of $+$'s that occur among the $c_i$'s is the same as the number of $-$'s that occur.  \\
\indent (d) The following symmetry holds: (i) 
$\text{if } c_i=\pm, \text{ then }c_{2n-i+1}=\mp \text{ and (ii) }
  \text{if }c_i=c_j\in\N, \text{ then }c_{2n-i+1}=c_{2n-j+1}\in\N.$ \\
Two clans are considered the same when they have $+$, $-$ and pairs of equal natural numbers in the same positions.  Note that the number of $+$'s plus the number of pairs of natural numbers is $n$.

An example of a clan for $n=8$ is $(+1-2+2+-|+-3-3+1-)$.

The clan encodes a $\gt$-stable Cartan subalgebra $\tilde{\fh}$, the action of $\gt$ on $\tilde{\fh}$ and a positive system $\tilde{\gD}^+$.  This data determines a $K$-orbit in $\fB$ having base point $\tilde{\fh}+\sum_{\tilde{\ga}\in\tilde{\gD}^+}\fg^{(\tilde{\ga})}$.  This is in fact a bijection between clans and $K$-orbits in $\fB$.  In Section \ref{subsec:sp} we fixed a Cartan subalgebra $\fh$ ($\subset\fk$) and a positive system $\gD^+$, and thus a Borel subalgebra $\fb=\fh+\fn$.  The $K$-orbit $K\cdot\fb$, which consists of all isotropic flags $\{0\}=F_0\subsetneq F_1\subsetneq\dots\subsetneq F_n=\C^n\times\{0\}$, corresponds to the clan $(++\cdots +|-\cdots--)$.

Given a clan, it is important for us to understand the action of $\gt$ on (simple) roots.  Suppose that $c$ corresponds to $\tilde{\fh},\tilde{\gD}^+$.  The simple roots in $\tilde{\gD}^+$ may be written as 
$$\tilde{\ga}_i=\tge_i-\tge_{i+1}, i=1,2,\dots,n-1, \text{ and }\tilde{\ga}_n=2\tge_n,
$$
for some basis $\{\tge_j\}$ of $\tilde{\fh}^*$.
Then the action of $\gt$ is given by
\begin{equation*}
\gt(\tge_i)=\begin{cases}
  \tge_i,  &\text{if }c_i=\pm  \\
  \tge_j,  &\text{if }c_j=c_i \text{ and }j\leq n  \\
  -\tge_{2n-j+1},  &\text{if }c_j=c_i\text{ and }j\geq n+1.
  \end{cases}
\end{equation*}    
For the example of the clan given above, this is 
$$ \gt(\tge_2)=-\tge_2, \gt(\tge_4)=\tge_6,\text{ and } \gt(\tge_6)=\tge_4,\text{ and }\gt(\tge_i)=\tge_i, \text{for }i=1,3,5,7,8.
$$
This gives
\begin{align*}
&\gt(\tilde{\ga}_1)=\tge_1+\tge_2, \; \gt(\tilde{\ga}_2)=-\tge_2-\tge_3, \;  \gt(\tilde{\ga}_3)=\tge_3-\tge_6,  \; \gt(\tilde{\ga}_4)=\tge_6-\tge_5, \\ 
&\gt(\tilde{\ga}_5)=\tge_5-\tge_4, \;  \gt(\tilde{\ga}_6)=\tge_4-\tge_7, \;\; \;\,   \gt(\tilde{\ga}_7)=\tge_7-\tge_8,\hspace{5pt}\gt(\tilde{\ga}_8)=2\tge_8.
\end{align*}
Recall that a root is called 
\begin{align*}
&\text{complex if }\gt(\ga)\neq\pm\ga  \\
&\text{real if }\gt(\ga)=-\ga  \\
&\text{compact imaginary if }\gt(\ga)=\ga \text{ and }\fg^{(\ga)}\subset \fk \\
&\text{noncompact imaginary if }\gt(\ga)=\ga \text{ and }\fg^{(\ga)}\subset \fp 
\end{align*}

For a simple root $\ga$ one may consider the generalized flag variety $\fF_\ga$ of all parabolic subalgebras of $\fg$ conjugate to $\tilde{\fh}+\fg^{(-\ga)}+\sum_{\gb\in\tilde{\gD}^+}\fg^{(\gb)}$.  Let $\pi_\ga:\fB\to\fF_\ga$ be the natural projection (with fiber $\C P_1$).  Section 5 of \cite{VoganIC3}, in particular Lemma 5.1, tells us that when $\fQ\leftrightarrow (\tilde{\fh},\tilde{\gD}^+)$ and 
\begin{equation}\label{eqn:weak}
\begin{split}
  &\ga\text{ is complex and }\gt(\ga)>0  \text{ or }  \\
  &\ga\text{ is noncompact imaginary}
\end{split}
\end{equation}
then $\pi_\ga^{-1}\pi_\ga(Q)$ contains a dense $K$-orbit of dimension $\dim(\fQ)+1$.  Following \cite[\S5]{VoganIC3} we denote this orbit by
\begin{equation}\label{eqn:o}
s_\ga\circ\fQ.
\end{equation}
Therefore, when (\ref{eqn:weak}) holds, we have an `operation' passing from a $K$-orbit $\fQ$ to a $K$-orbit $s_\ga\circ\fQ$ of one higher dimension.

This operation can be expressed in terms of clans as follows.  Suppose that $\fQ\leftrightarrow (\tilde{\fh},\tilde{\gD}^+)\leftrightarrow c$.  Then included in the cases for which (\ref{eqn:weak}) is satisfied are the  simple roots $\tilde{\ga}_j, j=1,2,\dots,n$ so that $(c_j,c_{j+1})$ is
\begin{align*}
\text{(i) } &(\pm,\mp)  \text{ or } \\
\text{(ii) } &(\pm,k) \text{ with the other $k$ occuring to the right.}
\end{align*}
In these cases, $s_{\tilde{\ga}_j}\circ \fQ_c$ is defined and is equal to the $K$-orbit with clan the same as $c$ except the $j$ and $j+1$ places are $(k\,k)$, in case (i) and are switched (as are the $n-j+1$ and $n-j$ places), in case (ii).

Let us consider $c_0=(+\cdots++|--\cdots-)$.  The only allowed operation (\ref{eqn:o}) is 
\begin{equation*}
s_{\tilde{\ga}_n}\circ c_0=(+\dots+1|1-\cdots-).
\end{equation*}
Now $s_{\tilde{\ga}_{n-1}}$ gives the only allowed operation; we get
\begin{equation*}
s_{\tilde{\ga}_{n-1}}\circ(s_{\tilde{\ga}_n}\circ c_0)=(+\dots+1+|-1-\cdots-).
\end{equation*}
Two operations are now allowed:
\begin{align*}
&s_{\tilde{\ga}_{n-2}}\circ s_{\tilde{\ga}_{n-1}}\circ s_{\tilde{\ga}_n}\circ c_0=(+\dots+1++|--1-\cdots-) \\
&s_{\tilde{\ga}_{n}}\circ s_{\tilde{\ga}_{n-1}}\circ s_{\tilde{\ga}_n}\circ c_0=(+\dots+1\hspace{1pt}2|2\hspace{1pt}1-\cdots-).
\end{align*}
Continuing we see that by applying all allowable operations (\ref{eqn:o}) to $c_0$ we generate all clans having the numbers $1,2,\dots,k$ (for any $k=0,1,\dots,n$) occurring left of center and the remaining $n-k$ slots left of center filled with $+$'s.  Note that if $c$ is of this type, then $s_{\tilde{\ga}_j}\circ c$ is defined if either (1) the $j,j+1$ slots are $+,m$, in which case $s_{\tilde{\ga}_j}\circ (\cdots +\hspace{1pt}m\cdots |\cdots)=(\cdots \hspace{1pt}m\hspace{1pt}+ \cdots |\cdots)$ or (2) $j=n$ and $s_{\tilde{\ga}_n}\circ (\cdots +|-
\cdots)=(\cdots\hspace{1pt} m|m\hspace{1pt}\cdots)$.


\subsection{}\label{subsec:bo} The support of the localization of a highest weight Harish-Chandra module is both a Schubert variety and the closure of a $K$-orbit in $\fB$.  It is shown in \cite[Appendix]{BZ15} that a Schubert variety $Z_w$ is the closure of a $K$-orbit if and only if it is the support of a highest weight Harish-Chandra module (i.e., $-w\rho$ is $\gD_c^+$ dominant, so $w\in\cW$).  Therefore, if
\begin{equation}\label{eqn:cl-1}
\Cl:=\{\text{clans }c: \bar{\fQ_c}\text{ is the support of a highest weight HC module}\},
\end{equation}
there is a bijection $\cW\leftrightarrow\Cl$ satisfying $w\leftrightarrow c$ if and only if $Z_w=\bar{\fQ_c}$.  Proposition \ref{prop:w-c} below gives this bijection explicitly.

As we have seen in \S\ref{subsec:clans}, the clan $c_0=(++\dots+|-\dots --)$ corresponds to the support of the holomorphic discrete series: $\bar{\fQ}_{c_0}=K\cdot \fb$, for our fixed choice of $\fb$.  As mentioned above, the clans obtained by applying all operations (\ref{eqn:o}) starting with $c_0$ have all $+$'s left of center and the pairs of natural numbers occur symmetrically about the center.  Note that there are $2^n$ such clans;  as we will see, these are precisely the clans of $\Cl$.  Since we will only be concerned with these clans we use the following shorthand notation.
\begin{Notation}
We write only the left half of the clan.  So, for example, the clan $(+12++|--21-)$ will be written as $c=(+12++)$. (This is not to be confused with the clan of size $n$ in \cite{Yamamoto97}.)
\end{Notation}
\begin{Prop}\label{prop:w-c}  The bijection $\cW\leftrightarrow\Cl$ for which $w\leftrightarrow c$ means $Z_w=\bar{\fQ}_c$ is given as follows.  The clan corresponding to $w\in\cW$ is obtained by replacing all positive entries of $w$ by $+$'s and all $-a$ (for $a>0$) by $a$.
\end{Prop}
\noindent For example, $w=(5\text{-}1\text{-}243)\leftrightarrow c=(+12++).$

To prove the proposition we first note the $c_0=(++\dots+)\leftrightarrow w_0=(n\dots21)$, then we apply all operations (\ref{eqn:o}) to both Weyl group elements and to clans.

The analogue of the operations (\ref{eqn:o})  for Schubert varieties is as follows.  Let $\ga\in S$, the set of simple roots, determine the generalized flag variety $\fF_\ga$ and the natural quotient map $\pi_\ga:\fB\to\fF_\ga$ as in \S\ref{subsec:clans}.  Then  $w\ga>0$ if and only if $\ell(ws_\ga)=\ell(w)+1$ and one easily checks that this happens precisely when $\pi_\gs^{-1}\pi_\ga(Z_w)=Z_{ws_\ga}$.  Observe that for $w\in \cW$,  a simple root $\ga_j=\ge_j-\ge_{j+1}$, $j=1,2,\dots,n-1$, satisfies $w\ga_j>0$ if and only if $w=(\dots\, a\,\text{-}b\,\dots)$, i.e., $w(j)=a, w(j+1)=-b$, for some $a,b>0$ (and necessarily $a>b$, as $w\in\cW).$  For the simple root $\ga_n=2\ge_n$, $w\ga_n>0$ if and only if $w=(\dots a)$, i.e., $w(n)=a>0$. 

It follows that if $w\leftrightarrow c$ as in the statement of the proposition, then for a given $\ga\in S$ the analogue of the  operation  (\ref{eqn:o}) for $W$ is defined if and only if it is defined for $c$.  Also, when the operation is defined for $\ga_j$ then the results $\tilde{w}=ws_j$ and $\tilde{c}$ correspond as in the statement:
\begin{align*}
   w&=(\dots a\,\text{-}b\dots),\;\tilde{w}=(\dots \text{-}ba\dots)  \\
   c&=(\dots +b\dots), \; \tilde{c}=(\,\dots\, b+\dots), \\
   \intertext{or}
      w&=(\dots a),\;\tilde{w}=(\dots \text{-}a)  \\
   c&
   =(\dots +), \; \tilde{c}=(\dots a).
\end{align*}
Note that  $Z_{\tilde{w}}=\pi_\gs^{-1}\pi_\ga(Z_w)=\pi_\gs^{-1}\pi_\ga(\bar{\fQ}_c)=\bar{\fQ}_{\tilde{c}}$, and $\cW$ is preserved under the operations (\ref{eqn:o}) .   The operations for $W$ applied to $w_0$ give all $2^n$ elements of $\cW$ and all of   $\Cl$ is generated by operations  (\ref{eqn:o}) beginning with $c_0$.    \hfill{$\square$}

\begin{Cor} $\displaystyle{\Cl=\{c:\text{ all entries of $c$ are +'s or natural numbers}\}.}$
\end{Cor}

\begin{Cor}\label{cor:closure}
Let $c_1,c\in\Cl$ and $c_1\leftrightarrow\fQ_1, c\leftrightarrow \fQ$.  Then  $\fQ_1\subset\bar{\fQ}$ if and only if there is a sequence of $s_\ga$ operations taking $\fQ_1$ to $\fQ$.
\end{Cor}

\begin{proof} This is now a restatement of Cor.~\ref{cor:order}.
\end{proof}

In \cite[Appendix]{BZ15} it is shown that an orbital variety $\bar{B\cdot\fn\cap\fn^w}$ is a $K$-orbit in $\fp_+$ if and only if $-w\rho$ is $\gD_c^+$ dominant.  Thus, the moment map images may be described by
\begin{equation*}
\mu(\bar{T_{\fQ}^*\fB})=\bar{B\cdot\fn\cap\fn^w},\text{ when }w\leftrightarrow c\leftrightarrow \fQ.
\end{equation*}
\begin{Prop}
Suppose $c_1,c\in\Cl$ and $\fQ_{c_1}\subset\bar{\fQ_c}.$  Then $\mu(\bar{T_{\fQ_{c_1}}^*\fB})\supseteq \mu(\bar{T_{\fQ_c}^*\fB})$.
\end{Prop}\label{prop:weak}
\begin{proof} It suffices to show that if $w_1,w\in\cW$ with $w_1\ga>0$ ($\ga\in S$) and $w=w_1s_\ga$, then $\fn\cap\fn^{w_1}\supseteq \fn\cap\fn^w$.  For this note that if $\gb\in\gD(\fn\cap\fn^w)$, then $\gb>0$ and $s_\ga w_1^{-1}\gb=w^{-1}\gb>0$,  so $w_1^{-1}\gb>0$ (since $w_1^{-1}\gb\neq\ga$).  Therefore, $\gb\in \gD(\fn\cap\fn^{w_1})$. 
\end{proof}

This fact will narrow down considerably the possible conormal bundle closures occurring in a characteristic cycle.  The inclusion of moment map images of this proposition is not true for arbitrary clans $c_1,c$, but depends on $c_1,c$ being in  $\Cl$.

Our induction from $(G',K')$ to $(G,K)$ will require that, given $c=(1 c')$ or $(+\,c')$ and $w\leftrightarrow c$ as in Proposition \ref{prop:w-c}, we relate $w'\leftrightarrow c'$ to $w$. Here $c'\in\Cl'$, with $\Cl'$ defined as $\Cl$ is defined, but for the pair $(G',K')$.  We relate $w'$ to $w$ as follows.  First, the Weyl group $W'$ is identified with a subgroup of $W$ in the natural way:
\begin{equation*}
 W'=\{w\in W: w(1)=1\}.
\end{equation*}
Restating Proposition \ref{prop:w-c} for $G'$ we have the following correspondence.  For $c'\in\Cl'$, express $c'$ as $+$'s and natural numbers $2,3,\dots,k$.  Now $w'$ begins with $1$ and the remaining entries are obtained from $c'$ by replacing $2,3,\dots,k$ by $-2,-3,\dots,-k$ and replacing $+$'s by $n,n-1,\dots,k+1$ (left to right).

\begin{Lem}\label{lem:ind}
Suppose $w\leftrightarrow c$ under the bijection $\cW\leftrightarrow \Cl$ of Proposition \ref{prop:w-c} and $w'\leftrightarrow c'$ under the bijection $\cW'\leftrightarrow \Cl'$ described above. \\
\indent (i) If $c=(1\,c)$, then $w=s_{2\ge_1}w'$, where $s_{2\ge_1}$ is the reflection in the root $2\ge_1$.  \\
\indent (ii) If $c=(+\,c')$, then $w=\gs w'$, where $\gs=s_{n-1}\dots s_2s_1$, $s_j$ the simple reflection in $\ga_j$.
\end{Lem}
The proof of this is immediate.

Here is an  example in each case.  (i) If $c=(12+34++)$, then $w=(\text{-}1\text{-}27\text{-}3\text{-}465)$. Then $c'=(2+34++)$ is in the correct form to write down $w'=(1\text{-}27\text{-}3\text{-}465)$.  Since $s_{2\ge_1}=(\text{-}1234567)$, it is clear that $w=s_{2\ge_1}w'$.  (ii) If $c=(+1+23++)$, then $w=(7\text{-}16\text{-}2\text{-}354)$.  Then $c'=(1+23++)$ is rewritten as $c'=(2+34++)$ and $w'=(1\text{-}27\text{-}3\text{-}465)$.  Since $\gs=(7123456)$, we have $w=\gs w'$.


\subsection{}\label{subsec:tab}
The $\rm{\mathbf {T}}_{\ga\gb}$ `operators' of \cite[Def.~3.4]{Vogan79b} will be used in \S\ref{subsec:CC} and \S\ref{subsec:harder} to help determine characteristic cycles.  Here we give the necessary background.  We also translate the description of the $\rm{\mathbf {T}}_{\ga\gb}$'s into clan notation.  

The $\tau$-invariant of $w$ is 
$$\tau(w):=\{\ga\in S: w\ga<0\}.
$$ 
Following \cite{Vogan79b}, given $\ga,\gb$ consecutive simple roots in the Dynkin diagram, we say that $w$ is in the domain of $\rm{\mathbf {T}}_{\ga\gb}$ when $\ga\notin\tau(w)$ and $\gb\in\tau(w)$.  When this is the case and $\ga$ and $\gb$ have the same length, then 
\begin{align}\label{eqn:tab-w}
  & \rm{\mathbf {T}}_{\ga\gb}(w)=\begin{cases} ws_\ga, & \gb\notin\tau(ws_\ga)  \\
                                ws_\gb, & \ga\in\tau(ws_\gb)
                  \end{cases}
\intertext{(where exactly one of the two possibilities occurs) and if $\ga,\gb$ have different root lengths}\notag
  &\rm{\mathbf {T}}_{\ga\gb}(w) =\left\{\tilde{w}:\tilde{w}\in\{ws_\ga,ws_\gb\},\ga\in\tau(\tilde{w}),\gb\notin\tau(\tilde{w})\right\}               
\end{align}
(a set of either one or two elements).

In Section 5 we will use the following important fact about the $\rm{\mathbf {T}}_{\ga\gb}$ operators and coherent continuation.  If $w$ is in the domain of $\rm{\mathbf {T}}_{\ga\gb}$, then 
\begin{equation}\label{eqn:coherent-c}
s_\ga(L_w)=L_w+L_y+\sum a_uL_u
\end{equation}
when either (i) $\ga,\gb$ are of equal length and $\rm{\mathbf {T}}_{\ga\gb}(w)=y$ or (ii) $\ga,\gb$ are of different length and $\rm{\mathbf {T}}_{\ga\gb}(w)=\{y\}$.  In the unequal length case and $\rm{\mathbf {T}}_{\ga\gb}(w)=\{y',y''\}$, then $L_y$ is replaced by $L_{y'}+L_{y''}$.  This fact is well-known and can be found in \cite{Jantzen79}.

The theory of coherent continuation and the $W$-equivariance of the characteristic cycle map (\cite{KashiwaraTanisaki84}) gives the following statement about the characteristic cycle of an irreducible highest weight module.    Write $CC(L_w)=\sum_y m_{y,w}[\bar{T_{B_y}^*\fB}].$

\begin{Lem}\label{lem:mcgovern}{\upshape (\cite[Lemma 3.1]{Mcgovern00})}  
Suppose $y,w$ are in the domain of some $\rm{\mathbf {T}}_{\ga\gb}$ and suppose $m_{y,w}\neq 0$.  Then
\newline\indent (i) if $\ga,\gb$ have the same lengths, then $m_{_{\rm{\mathbf {T}}_{\ga\gb}(y),\rm{\mathbf {T}}_{\ga\gb}(w)}}\neq 0$, and 
\newline\indent (ii) if $\ga,\gb$ have different lengths, then for $y'\in \rm{\mathbf {T}}_{\ga\gb}(y)$ there is $w'\in \rm{\mathbf {T}}_{\ga\gb}(w)$ so that $m_{y',w'}\neq 0$.
\end{Lem}

We may conclude the following.
\begin{Lem}\label{lem:mcgovern+}
If $y,w$ are in the domain of some $\rm{\mathbf {T}}_{\ga\gb}$ and  $\ga,\gb$ have the same lengths, then $m_{y,w}=m_{_{\rm{\mathbf {T}}_{\ga\gb}(y),\rm{\mathbf {T}}_{\ga\gb}(w)}}$.
\end{Lem}
\begin{proof} The proof of McGovern's lemma gives $m_{y,w}\geq m_{_{\rm{\mathbf {T}}_{\ga\gb}(y),\rm{\mathbf {T}}_{\ga\gb}(w)}}$.  Since $\rm{\mathbf {T}}_{\ga\gb}(y)$ and $\rm{\mathbf {T}}_{\ga\gb}(w)$ are in the domain of $\rm{\mathbf {T}}_{\gb\ga}$ and $\rm{\mathbf {T}}_{\gb\ga}\rm{\mathbf {T}}_{\ga\gb}(y)=y,\rm{\mathbf {T}}_{\gb\ga}\rm{\mathbf {T}}_{\ga\gb}(w)=w$, we conclude that the other inequality also holds.
\end{proof}

Using the above definition of $\rm{\mathbf {T}}_{\ga\gb}$ along with Proposition \ref{prop:w-c} we may translate the above formula for  the $\rm{\mathbf {T}}_{\ga\gb}$'s (\ref{eqn:tab-w}) into clan notation.

\begin{Notation}
For $\ga,\gb$ two consecutive roots $\ga_i,\ga_{k}$, we write $\rm{\mathbf {T}}_{\ga\gb}$ as $\rm{\mathbf {T}}_{j\,k}$ ($k=j\pm 1$).  Suppose $w\in\cW$ and $w\leftrightarrow c$ (as in Proposition \ref{prop:w-c}), then we write $\rm{\mathbf {T}}_{jk}(c)$ for the clan (or pair of clans) corresponding to the Weyl group element, or pair, $\rm{\mathbf {T}}_{\ga\gb}(w).$  We also write $\tau(c)$ for $\tau(w)$.
\end{Notation}
 
We have $\ga_j=\ge_j-\ge_{j+1}\in\tau(c)$ if and only if the $j$ and $j+1$ entries of the clan $c$ are $++$, $k+$ or $k\,k\!+\!1$, and $\ga_n=2\ge_n\in\tau(c)$ when the last entry of $c$ is a natural number.
 
Using the abbreviated notation of $\rm{\mathbf {T}}_{j\,j+1}$ for $\rm{\mathbf {T}}_{\ga_j\ga_{j+1}}$ and writing the $j,j+1,j+2$ entries of the clans, in the case of equal root lengths we have:
\begin{align}\begin{split}\label{eqn:tabc}
&\rm{\mathbf {T}}_{j\,j+1}( \,\cdots +\;\,k\,\;+ \,\cdots)=(\cdots \,\;+\,\;+\,\,\,k \,\cdots)  \\
&\rm{\mathbf {T}}_{j\,j+1}( \,\cdots +\,k{\text{-}}1\,\,k \,\cdots)=(\cdots \,k{\text{-}}1\,+\,\,k \,\cdots)  \\
&\rm{\mathbf {T}}_{j+1\,j}( \,\cdots +\;\,+\;\;\,k \,\cdots)=(\cdots \;\,+\,\,\;k\,\,+ \,\,\cdots)  \\
&\rm{\mathbf {T}}_{j+1\,j}( \,\cdots k{\text{-}}1\,+\,k \,\cdots)=(\cdots \;\,+\;k{\text{-}}1\,\,k \,\,\cdots). \end{split} \\
\intertext{For the last two simple roots} 
\begin{split}\label{eqn:tabcc}&\rm{\mathbf {T}}_{n\; n-1}(\dots +\,+)=(\dots + \, (k\!{\text{\footnotesize+}}\!1))  \\
&\rm{\mathbf {T}}_{n\; n-1}(\dots \,\,k \;\,+)=(\dots + \,\,\,k\,)  \\
&\rm{\mathbf {T}}_{n-1\; n}(\dots + \,k)=\left\{(\dots \,k \,+), (\dots + +)\right\}.
\end{split}\end{align}
As an example, consider the clan $c=(++\dots ++)$.  Then $\ga_n\notin\tau(c)$, $\ga_{n-1}\in\tau(c)$, and $\rm{\mathbf {T}}_{n\;n-1}(++\dots ++)=(++\dots +1)$.  Now, formula (\ref{eqn:coherent-c}) gives
\begin{equation*}\label{eqn:coherent-c-c}
s_n(X_{(++\dots++)})=X_{(++\dots++)}+X_{(++\dots+1)}+\sum X_d,
\end{equation*}
where $\ga_{n-1},\ga_n\in\tau(d)$.


\section{Harish-Chandra cells}\label{sec:hc-cells}  The Harish-Chandra cells of highest weight Harish-Chandra modules are determined in this section.  They are given by an inductive procedure in terms of clans.  The statement is contained in Proposition \ref{prop:hc-cells}.  We will continue our notion of $(G,K)$ for $(Sp(2n),GL(n))$, $\Cal O_k$ the $K$-orbit in $\fp_+$ of (\ref{eqn:Korbit}), $\fb=\fh+\fn$ for the Borel subalgebra of $\fg$ as in \S\ref{subsec:sp}, etc.  We use the notation $(G',K')$ for $(Sp(2(n-1)),GL(n-1))$, $\Cal O_k'$ for a $K'$-orbit in $\fp_+'=\fp_+\cap \fg'$, etc.

\vspace{4pt}
\begin{Notation}
We use the shorthand notation $T_c$ for $\bar{T_{\fQ_c}^*\fB}$, when $c$ is the clan parametrizing the $K$-orbit $\fQ_c$.
\end{Notation}


\subsection{} We begin by describing the moment map images of the closures of the conormal bundles to $K$-orbits  in $\fB$.  Each $c\in \Cl$ is of the form $(1\,c')$ or $(+\hspace{1pt}c')$ with $c'\in\Cl'$, the set of clans for $(G',K')$ as in (\ref{eqn:cl-1}).  The following lemma gives $\mu(T_c)$ in terms of $\mu(T_{c'})$.  This lemma follows from \cite[\S3.5]{Yamamoto97}.
\begin{Lem} Let $c'\in\Cl'$ and suppose that $\mu(T_{c'})=\bar{\Cal O_k'}$.  Then 
\begin{align*}
   &\text{\upshape{(i)}}\,\,\,  \mu(T_{(1\,c')})=\bar{\Cal O}_k  \\
   &\text{\upshape{(ii)}} \,\,\mu(T_{(+\hspace{1pt}c')})=
     \begin{cases} \bar{\Cal O}_{k+2}, &  k\leq n-2  \\
                   \bar{\Cal O}_n, &  k=n-1.
     \end{cases}
\end{align*}     
\end{Lem}
\begin{proof} 
Recall that if $c\leftrightarrow w$ as in Proposition \ref{prop:w-c}, then $\mu(T_c)=\bar{B\cdot\fn\cap\fn^w}\subset\fp_+\simeq\sym(n)$.  This is $\bar{\Cal O_l}$ exactly when $\fn\cap\fn^w$ contains a matrix of rank $l$, but no matrix of higher rank.  Suppose $c'\leftrightarrow w'\in W'=W(C_{n-1})$.

Consider $c=(1\,c')$.  Then $c\leftrightarrow w=s_{2\ge_1}w'$, by Lemma \ref{lem:ind}.  It follows that $\gD(\fn\cap\fn^w)=\gD(\fn'\cap\fn'^{\,w'})$.  Check: $\gD(\fn)\smallsetminus\gD(\fn')=\{\gb\in\gD^+:\IP{\gb}{\ge_1}>0\}$ and $s_{2\ge_1}\gb=\gb$, for $\gb\in\gD(\fn')$.  Now $\gb\in\gD(\fn\cap\fn'^{\,w'}) \iff \gb\in\gD(\fn')$ and $w'^{-1}(\gb)>0 \iff \gb\in\gD(\fn)$ and $w'^{-1}s_{2\ge_1}\gb>0 \iff \gb\in\gD(\fn\cap\fn^w).$  Therefore, (i) holds.

Now consider $c=(+\,c')$.  Then, by Lemma \ref{lem:ind}, $c\leftrightarrow w=\gs w'$, $\gs=s_{n-1}\dots s_2s_1$, where $s_i$ is the simple reflection for the $i^\text{th}$ simple root.  Then 
\begin{equation}\label{eqn:ii}
\gs^{-1}\gD(\fn\cap\fn^w)=\{\ge_1+\ge_j:j=1,2,\dots,n\}\cup\gD(\fn'\cap\fn'^{\,w'}).
\end{equation}
We check this.  Since $\gs\in W_K$, $\gs(\gD(\fp_+))=\gD(\fp_+).$  Therefore, $\ga\in\gs^{-1}\gD(\fn\cap\fn^w)$ $\iff$ $\gs(\ga)>0$ and $w^{-1}\gs(\ga)>0$ $\iff$ $\ga>0$ and $w'^{-1}\ga>0$.  It follows , since $w'$ permutes $2,3,\dots,n$, that any $\ga$ in the right-hand side of (\ref{eqn:ii}) is in $\gs^{-1}\gD(\fn\cap\fn^w)$.  On the other hand, suppose $\ga\in\gs^{-1}\gD(\fn\cap\fn^w)$. Then either $\ga\in\gD(\fn)\!\smallsetminus\!\gD(\fn')$ (so $\ga=\ge_1+\ge_j$) or $\ga\in\gD(\fn')$ (so $\ga\in\gD(\fn'\cap\fn'^{w'})$, since $w'^{-1}(\ga)\in\fn\cap\fg'=\fn'$).  Therefore $\ga$ is in the right-hand side.

Since $\mu(T_{(+\hspace{1pt}c')})=\bar{B\cdot\fn\cap\fn^w}$, we need to determine the maximum rank of a matrix in $\fn\cap\fn^w\subset\sym(n)$.  By (\ref{eqn:ii}) $\fn\cap\fn^w$ is the set of matrices
$$X=\begin{bmatrix} x_1 & \cdot&\cdot & x_n  \\
                        &  & & \cdot     \\
                        &  {X'} & & \cdot      \\
                        &  & & x_1
    \end{bmatrix},\;\; x_1,\dots,x_n\text{ arbitrary  and }\rank(X')\leq k.
$$
The rank of $X$ is at most $k+2$.  When $\rank(X')=k$ and $k\leq n-2$, then $X=X'+X_{\ge_j}+X_{\ge_n}$ will have rank $k+2$ for some $j=1,\dots,n-1$ (by choosing $j$ so that the $j$th column of $X'$ is a linear combination of the other columns). When $\rank(X')=k=n-1$ (so no such $j$ exists), $X=X'+X_{2\ge_n}$ has rank $n$.
\end{proof}

In analogy with geometric cells for Weyl groups, we set 
$$\Cg(k):=\{c\in\Cl: \mu(T_c)=\bar{\Cal O}_k\}.$$
By Appendix \ref{app:B}, $\Cg(k)$ is the set of \emph{all} clans for $(G,K)$ with $\mu(T_c)=\bar{\Cal O}_k.$  Restating the lemma slightly we have
\begin{equation}\label{eqn:g-cells}
\begin{split}
&\Cg(n)=\{(+\hspace{1pt}c') : c'\in\Cg'(n-1)\cup\Cg'(n-2) \}\text{ and }  \\
&\Cg(k)=\{(+\hspace{1pt}c') : c'\in\Cg'(k-2)\}\cup \{(1\,c') : c'\in\Cg'(k)\}, 2\leq k\leq n-1, \\
&\Cg(k)=\{(1\,c') : c'\in\Cg'(k)\}, k=0,1.
\end{split}
\end{equation}
A simple counting argument gives the size of each $\Cg(k)$.
\begin{Prop}\label{prop:34} $\displaystyle{^\#\!\Cg(k)={n\choose [\frac{k}{2}]}}.$
\end{Prop}
\begin{proof} We use induction on $n$ .  For $n=1$ there are just two clans: $\mu(T_{(+)})=\bar{\Cal O}_1$ and $\mu(T_{(1)})=\bar{\Cal O}_0$, and the statement holds.  For $n>1$,
\begin{align*}
 &^\#\!\Cg(k)=^\#\!\!\Cg'(k)+^\#\!\!\Cg'(k-2)={n-1\choose [\frac{k}{2}]}+{n-1\choose [\frac{k}{2}]-1}={n\choose [\frac{k}{2}]}, 2\leq k\leq n-1, \\ 
 &^\#\!\Cg(n)=^\#\!\!\Cg'(n-1)+^\#\!\!\Cg'(n-2)={n-1\choose [\frac{n-1}{2}]}+{n-1\choose [\frac{n-2}{2}]}={n\choose [\frac{n}{2}]},   \\
 &^\#\!\Cg(k)=^\#\!\!\Cg'(k)=1, k=0,1.
\end{align*} 
\end{proof}


\subsection{} For arbitrary irreducible Harish-Chandra modules, if $AV(X)=\bar{\Cal O}^1\cup\dots\cup\bar{\Cal O}^m$, then the $G$-saturations of the $\Cal O^i$ coincide and the closure of this complex orbit $\Cal O^\C$ is the associated variety of the annihilator of $X$.  See \cite{Vogan91}.  A complex orbit closure $\bar{\Cal O^\C}$ is the associated variety of the annihilator of some irreducible Harish-Chandra module of infinitesimal character $\rho$ if and only if $\Cal O^\C$ is `special' in the sense of \cite{Lusztig79}.  This fact follows from the results of \cite{BarbaschVogan82}. The special orbits are known and are listed, for example, in \cite{CollingwoodMcgovern93}

The orbits of interest to us are the $\Cal O_k,k=0,1,\dots,n$.  An orbit $\Cal O_k^\C$ is special if and only if $k$ is even or $k=n$.

To each complex nilpotent orbit $\Cal O^\C$ is associated an irreducible $W$-representation $\pi(\Cal O^\C)$ (\cite{Springer78}).  There is a procedure for the classical groups  (described in \cite[\S10.1]{CollingwoodMcgovern93}) of associating to the partition parametrizing a nilpotent orbit, a pair of partitions parametrizing the Weyl group representation $\pi(\Cal O^\C)$ in terms of Lusztig's symbols (\cite{Lusztig79}).  For the orbits $\Cal O_k^\C$  one sees that 
\begin{equation}\label{eqn:sp-dim}
\dim(\pi(\Cal O_k))={n\choose [\frac{k}{2}]}, k=0,1,\dots,n.
\end{equation}

Harish-Chandra cells are defined in \cite{BarbaschVogan83}; they partition the irreducible Harish-Chandra modules of infinitesimal character $\rho$.  It follows easily from the definitions that any two irreducible representations in the same Harish-Chandra cell have the same associated variety.  It also follows that if a Harish-Chandra cell $\Chc$ contains one highest weight Harish-Chandra module then $\Chc$ consists entirely of highest weight Harish-Chandra modules.  By our discussion of special orbits, we know that the only nilpotent $K$-orbits occurring as  associated variety for a Harish-Chandra cell of highest weight Harish-Chandra modules (infinitesimal character $\rho$) are $\Cal O_k$, with $k$ even or $k=n$.  It is also a fact (Appendix \ref{app:B}) that if irreducible representations in a Harish-Chandra cell have associated variety $\bar{\Cal O}_k$, then the cell consists of highest weight Harish-Chandra modules. 

A Harish-Chandra cell $\Chc$ determines a representation of $W$ on a subquotient of the coherent continuation representation (\cite{BarbaschVogan83}).  This representation is spanned by $\Chc$ and will be denoted by $V(\Chc)$.  If $\bar{\Cal O^\C}$ is the associated variety of the annihilator of representations in $\Chc$, then the special representation $\pi(\Cal O^\C)$ occurs in $V(\Chc)$ as a $W$-submodule (\cite[Cor.~14.11]{VoganIC4}).  The results of \cite{McGovern98} tell us much more.  Theorem 7 of \cite{McGovern98} states that all Harish-Chandra cell representations  $V(\Chc)$ for $Sp(2n,\R)$ are of the particular type, called `Lusztig'.  It follows that the decomposition of $V(\Chc)$ into irreducible $W$-representations may be determined explicitly in a combinatorial manner in terms of the symbols.  This procedure is given in \cite[Ch.~12]{Lusztig84} and is described in a convenient way in \cite[Thm.~2.12]{McGovern96}.

For the orbits $\Cal O_k$ one sees that if the associated variety of annihilator for $\Chc$ is $\bar{\Cal O}_k^\C$ (so $k$ is even or $k=n$) the $W$-representations $V(\Chc)$ are as follows.  If $k$ is even and nonzero, the symbol of $\pi(\Cal O_k^\C)$ is
\begin{equation*}
\setcounter{MaxMatrixCols}{15}
\begin{pmatrix}
0 & & 1 & & 2 & &   \cdots &  n-k & & n-k+2 &  n-k+3 &  \cdots & & & \hspace{-18pt} n-\frac{k}{2} + 1 \\
& 1 & & 2 & & 3 & & & \cdots &    & &  &   n-\frac{k}{2} & 
\end{pmatrix}.
\end{equation*}
The `Lusztig' left cell is a sum of this $W$-representation and the one with symbol
\begin{equation*}
\begin{pmatrix}
0 & & 1 & & 2 & & &  \cdots &    &    &    & & \hspace{-18pt} n-\frac{k}{2} + 1 \\
& 1 & & 2 & &   \cdots & n-k  & & n-k+2 & \,\,n-k+3   & \cdots  &  n-\frac{k}{2} & 
\end{pmatrix}.
\end{equation*}
One easily checks that this symbol is the one for $\pi(\Cal O_{k-1}^\C)$.  If $k=n$ is odd, then $V(\Chc)$ is the irreducible representation $\pi(\Cal O_n^\C)$.  This gives the following proposition.

\begin{Prop}\label{prop:size}  
If $\Chc$ is a Harish-Chandra cell consisting of irreducible Harish-Chandra modules of associated variety $\bar{\Cal O}_k$ (so $k$ is even or $k=n$), then as $W$-representations
$$V(\Chc)=\begin{cases} \pi(\Cal O_k^\C) \oplus \pi(\Cal O_{k-1}^\C),  & \text{if $k>0$ is even}  \\
\pi(\Cal O_n^\C),  & \text{if $k=0$ or $k=n$, odd}
\end{cases}
$$  
and 
$$
^\#\Chc=\begin{cases}  {n\choose \frac{k}{2}}+ {n\choose \frac{k}{2}-1},  &k>0,k\text{ even},  \\
{n\choose \frac{n-1}{2}}, &k=n, k\text{ odd},  \\
\;1,  & k=0.
\end{cases}
$$
\end{Prop}


\subsection{} In this section we explicitly determine the Harish-Chandra cells of highest weight Harish-Chandra modules.  
\begin{Lem}\label{lem:exist}
For $k=0,1,\dots,n$ and either $k$ even or $k=n$, there is an irreducible highest weight Harish-Chandra module of infinitesimal character $\rho$ with associated variety $\Bar{\Cal O}_k$.
\end{Lem}
\begin{proof}  Let $\fQ_l$ be the $K$-orbit in $\fB$ corresponding to the clan $c_l=(+\dots+1\,2\dots n\text{-}l)$.  We first show that $\bar{\fQ_l}$ is smooth.  By Prop.~2.5, $\bar{\fQ_l}=Z_{w_l}$, $w_l=(n\,n\text{-}1\,\dots\, n\text{-}l\text{+}1\,\text{-}1\,\text{-}2\,\dots\,\text{-}(n\text{-}l))$.  It is easy to verify that $w_l$ has maximal length in its $W(C_{n-l})$ coset.  Let $\pi_l:\fB\to\fF_l$ be the natural projection, where $\fF_l$ is the generalized flag variety of parabolics conjugate to $\fq_l$ defined by $\gD(\fl)=\langle S\smallsetminus\{\ga_1,\dots,\ga_l\}\rangle$.  It follows that $Z_{w_l}=\pi_l^{-1}(\bar{Bw_l\cdot\fq_l})$.  So $Z_{w_l}$ is smooth if $\bar{Bw_l\cdot\fq_l}$ is smooth.  However, $w_l$ and the long element $w_{K,0}$ of $W(K)$ are in the same $W(C_{n-l})$ coset, so $\bar{Bw_l\cdot\fq_l}=\bar{Bw_{K,0}\cdot\fq_l}$.  Now
\begin{equation*}
K\cdot\fq_l\supset \bar{w_{K,0}B^{w_{K,0}^{-1}}\cdot\fq_l}
 =\bar{Bw_{K,0}\cdot\fq_l} 
 =\bar{Bw_{K,0} B\cdot\fq_l}
 \supset \bar{Bw_{K,0} B}\cdot\fq_l
 =K\cdot\fq_l,
\end{equation*}
It follows that $\bar{Bw_l\cdot\fq_l}=K\cdot\fq_l\simeq K/K\cap Q_l$ is smooth, so also $Z_{w_l}$.

By smoothness, $AV(X_{\fQ_{c_l}})=\mu(T_{c_l})$.  The Weyl group element corresponding to $c_l$ is
$$
w_l=(n\,n\text{-}1\dots n\text{-}l\text{\smaller+}1\;\text{-}1\;\text{-}2\dots \text{-}(n\text{-}l)),
$$ 
and $\mu(T_{c_l})=\bar{B\cdot\fn\cap\fn^{w_l}}$.  Since $\gD(\fn\cap\fn^{w_l})=\{\ge_i+\ge_j : 1\leq i\leq n,\,n-l+1\leq j\leq n\}$, a matrix in $\fn\cap\fn^{w_l}\subset\sym(n)$ has the form 
$$
\begin{bmatrix}
  A  &  0  \\
  C  &  S_lA^tS_{n-l}
\end{bmatrix},\,\,  
A\in M_{(n-l)\times l}, C\in \sym(l).
$$
The maximal rank of such a matrix is $\min\{2l,n\}$.  Therefore, all $\bar{\Cal O}_k,$ with either $k$ even or $k=n$ occur as associated varieties .
\end{proof}

\begin{Rem} The representations $X_{\fQ_{c_l}}$ in the above proof are the cohomologically induced representations $\Cal R_{\bar{\fq}_l}(\C)$, where $\fq_l$ is defined by $H_l:=\sum_{i=1}^l\ge_i$, ($\gD(\fq_l)=\{\ga : \IP{H_l}{\ga}\geq 0\}$).
\end{Rem}
\begin{Rem}  The lemma also follows from \cite[\S7]{Joseph92}.  It is proved there that the unitary spherical highest weight modules with highest weights $\gl_k=-\dfrac{k}{2}\gL_n$ ($\gL_n$ the $n^\text{th}$ fundamental weight) have associated variety $\bar{\Cal O}_k$, for each $k=0,1,\dots,n$.  The infinitesimal character $\gl_k+\rho$ is integral exactly when $k$ is even or $k=n$.  This infinitesimal character may be singular, but the translation principle says that, when $k$ is even, there exists a highest weight Harish-Chandra module with infinitesimal character $\rho$ and the same associated variety.
\end{Rem}

\begin{Cor} For each $k$ even and for $k=n$ (when $n$ is odd), $0\leq k\leq n$, there is exactly one Harish-Chandra cell $\Chc(k)$ consisting of irreducible Harish-Chandra modules of associated variety $\bar{\Cal O}_k$.  The sizes of these cells are
$$^\#\Chc(k)=\begin{cases}
       {n+1\choose \frac{k}{2} },  &  k \text{ even}  \\
       {n\choose \frac{n-1}{2}},  &  k=n\text{ odd}.
       \end{cases}
$$      
\end{Cor} 
\begin{proof} Existence is Lemma \ref{lem:exist} and the size is Proposition \ref{prop:size}.  The uniqueness follows from our comment that only highest weight Harish-Chandra modules can have associated variety $\bar{\Cal O}_k$ by checking that the $\Chc(k)$'s exhaust all $2^n$ highest weight Harish-Chandra modules.  For $n$ even
\begin{align*}
 \sum_{k\text{ even}}&^\#\Chc(k)=\sum_{l=0}^\frac{n}{2}{^\#}\Chc(2l)  \\
     &=^\#\Chc(0)+\sum_{l=1}^\frac{n}{2}{^\#}\Chc(2l)  \\
     &=1+\sum_{l=1}^{\frac{l}{2}}\left({n\choose l }+{n\choose l-1 }  \right)  \\
     &=\sum_{l=0}^n {n\choose l } \\
     &=2^n.
\end{align*}
When $n$ is odd one similarly gets
$$\sum_{k\text{ even}}{^\#}\Chc(k)+^\#\Chc(n)=2^n.$$
\end{proof}

The Harish-Chandra cells may now be described exactly.
\begin{Prop}\label{prop:hc-cells} The Harish-Chandra cells $\Chc(k)$ are as follows.
\begin{enumerate}
 \item $\Chc(0)=\{\C\}.$
 \item If $n$ is odd, $\Chc(n)$ consists of all highest weight Harish-Chandra module $X_\fQ$ of support $\bar{\fQ}$ with $\fQ$ having clan in $\Cg(n)$.
 \item If $k$ is even, $\Chc(k)$ consists of  all highest weight Harish-Chandra module $X_\fQ$ of support $\bar{\fQ}$ with $\fQ$ having clan in $\Cg(k)\cup\Cg(k-1)$.
\end{enumerate}
\end{Prop}
\begin{proof}
(1) is clear.  For (2), each $X_\fQ$ with $\fQ\in\Cg(n)$, has associated variety $\bar{\Cal O}_n$ (since this is already the moment map image of the conormal bundle for the support).  But $^\#\Chc(n)=^\#\!\!\Cg(n)$, so (2) follows.

We check (3) by applying induction to $k$ (even).  Note that 
\begin{align}\label{eqn:A}
&^\#\Chc(k)=^\#\!\!\Cg(k)+^\#\!\!\Cg(k-1),
\intertext{and}\label{eqn:B}
&\text{if }X_\fQ\in\Chc(k), \text{ then }\fQ\in\Cg(j)\text{ for some }j\leq k.
\end{align}
If $k=2$ the statement of (3) holds as no $X_\fQ, \fQ\in \Cg(j)$ with $j>2$ can have associated variety $\bar{\Cal O}_2$.  Therefore, $\fQ\in\Cg(2)\cup\Cg(1)$ (since $\fQ\notin\Cg(0)$).

Now assume (3) holds for $k<k'$ and suppose that $AV(X_\fQ)=\bar{\Cal O}_{k'}$.   Then $\fQ\in\cup_{j\leq k'}\Cg(j)$.  Since (by induction) all $\fQ\in\cup_{j\leq k'-2}\Cg(j)$ are accounted for, $\fQ$ must be in $\Cg(k')\cup\Cg(k'-1)$.
\end{proof}


\section{A Lemma}\label{sec:lemma}
The Riemann-Hilbert correspondence gives an equivalence of categories between $\Cal M_c(\Cal D_\fB,K)$ and a derived category of perverse sheaves.  The characteristic cycle of a perverse sheaf is defined in \cite{Kashiwara85}.  It is shown there that if $\fX\in \Cal M_c(\Cal D_\fB,K)$ and $RH(\fX)$ is the corresponding perverse sheaf, then $CC(\fX)=CC(RH(\fX))$.  If $X\in\Cal M(\fg,K)$ is irreducible and $\fX$ is the corresponding localization (and $\supp(\fX)=\bar{\fQ}$), then $RH(\fX)=\IC_{\bar{\fQ}}$, the intersection cohomology sheaf.  This complex of sheaves depends only on the topology of $\bar{\fQ}$.

\subsection{Normal slices} Let $w_1,w\in W$ with $Z_{w_1}\subset Z_w$.  Consider the affine open neighborhood 
$$U_{w_1}=w_1\bar{N}\cdot x_0, \; x_0=\fb\in\fB,
$$
of $x_{w_1}:=w_1\cdot x_0$ in $\fB$.  Here $\fb=\fh+\fn$ is the fixed Borel subalgebra of \S\ref{subsec:sp} and $\bar{N}=\exp(\bar{\fn}),$ $\fnbar:=\sum_{\ga\in\gD^+}\fg^{(-\ga)}$.
Then 
\begin{align*}
U_{w_1}\cap (B\cdot x_{w_1})&= w_1(\bar{N}\cap w_1^{-1}Bw_1)\cdot x_0  \\
&=(w_1\bar{N}w_1^{-1}\cap N)w_1\cdot x_0  \\
&=\exp(\bar{\fn}^{w_1}\cap\fn)\cdot x_{w_1}.
\end{align*}
Therefore,
$$S_1:=\exp(\bar{\fn}^{w_1}\cap\bar{\fn})\cdot x_{w_1}
$$
is a normal slice to $B\cdot x_{w_1}$ at $x_{w_1}$.  This means that (i) $S_1$ is smooth, (ii) $S_1\cap (B\cdot x_{w_1})=\{x_{w_1}\}$, and (iii) $T_{x_{w_1}}S_1 + T_{x_{w_1}}(B\cdot x_{w_1})=T_{x_{w_1}}\fB$.  We want to understand $Z_w\cap S_1$; this will tell us about the occurrence of $\bar{T}_{Z_{w_1}}$ in $CC(\IC_{Z_w})$.

Now assume that $w_1,w\in\cW$.  We consider two cases. \\
{\bf Case I.} $w\leftrightarrow c=(1\,c'),w_1\leftrightarrow (1\,c_1').$  \\
{\bf Case II.} $w\leftrightarrow c=(+\hspace{1pt}c'),w_1\leftrightarrow (+\hspace{1pt}c_1').$  \\
Let $w'\leftrightarrow c'$ and $w_1'\leftrightarrow c_1'$, then $w_1',w'\in\cW'$, as described in Lemma~\ref{lem:ind}.

In the following discussion we view $G'$ as a subgroup of $G$ as in \S\ref{subsec:sp}.  The flag variety of $\fB'$ is $G'\cdot (\fb\cap\fg')\simeq G'\cdot\fb$.  It follows that, on identifying $\fB$ with isotropic flags in $\C^n$, $\fB'$ consists of those isotropic flags $(F_i)$ so that $F_1=\C e_1$.

\begin{Lem}\label{lem:4b}  Let $w_1\in\cW$.  Then $\bar{\fn}\cap\bar{\fn}^{w_1},\bar{\fn}'\cap\bar{\fn}'^{w_1'}\subset \fp_-\simeq \sym(n)$. \\
\indent {\upshape(a)} In case I,  $\bar{\fn}\cap\bar{\fn}^{w_1}=\bar{\fn}'\cap\bar{\fn}'^{w_1'}$.  \\
\indent {\upshape(b)} In case II,  
$\bar{\fn}\cap\bar{\fn}^{w_1'}=\gs(\bar{\fn}'\cap\bar{\fn}'^{w_1})\cup \gs\left(\fp_-\!\smallsetminus\! \fp_-'\right).
$
\end{Lem}

\begin{proof}  In case I, Lemma~\ref{lem:ind} tells us that $w_1=s_{2\ge_1}w_1'$ and  $w=s_{2\ge_1}w'$.  Therefore,
\begin{align*}
  \ga\in &\gD(\fn\cap\fn^{w_1}) \iff  \ga>0 \text{ and } w_1^{-1}\ga >0  \\
   &\iff \ga>0 \text{ and } w_1'^{-1} s_{2\ge_1}\ga >0  \\
   &\iff \ga>0 \text{ and } w_1'^{-1}\ga >0, \ga\in\gD(\fn').
\end{align*}
The last equivalence holds since $w_1'^{-1}s_{2\ge_1}(\ge_1)=-\ge_1$ and so no root $w_1'(-\ge_1\pm\ge_j)$ is positive.

In case II, Lemma~\ref{lem:ind} says that $w_1=\gs w_1'$.  Let $\ga=\gs(\gb)\in \gs(\bar{\fn}'\cap\bar{\fn}'^{w_1})\cup \gs\left(\gD(\fn)\!\smallsetminus\! \gD(\fn')\right).$  Then $\ga>0$ and $w_1^{-1}\ga=w_1'^{-1}\gs^{-1}\ga=w_1'^{-1}\gb>0$.  This is clear when $\gb\in \bar{\fn}'\cap\bar{\fn}'^{w_1}$.  When $\gb\in\gD(\fn)\!\smallsetminus\!\gD(\fn')$, $w_1'^{-1}\gb=\ge_1\pm\ge_k>0$ (some $k$), since $w_1'(\ge_1)=\ge_1$.
\end{proof}

Appendix \ref{app:1} gives a description of the relevant Schubert varieties in terms of flags.  The statement (Proposition \ref{prop:sb}) is that for $a=(a_1,\dots,a_n)$  defined by
$$
a_i=\#\{j : w(j)>0\text{ and } j\leq i\},
$$
$Z_w=\{(F_i) : \dim(F_i\cap (\C^n\times\{0\}))\geq a_i, i=1,\dots n\}.$  Viewing $Z_{w'}\in\fB'$ as a flag in $\fB$,
$$
Z_{w'}=\{(F_i) : \dim(F_i\cap (\C^n\times\{0\}))\geq a_i', i=1,\dots n\},
$$
$a_i'=\#\{j : w'(j)>0\text{ and } j\leq i\}.$  From the form for $w'$ given in the proof of the above Lemma, we have 
\begin{equation}\label{eqn:as}
\begin{split}
  &\text{in case I, } a_i'=a_i+1,  \\
  &\text{in case II, } a_i'=a_{i}.
\end{split}
\end{equation}

The following is our main geometric lemma.  We write $S_1'$ for the slice in $\fB'$ defined by $w_1'$.  Therefore, we have the coordinates
\begin{equation}\label{eqn:xx}
\begin{split}
  & \varphi_1:\bar{\fn}\cap\bar{\fn}^{w_1}\to S_1, \varphi_1(X)=\exp(X)\cdot x_{w_1},  \\
  & \varphi_1':\bar{\fn}'\cap\bar{\fn}'^{w_1'}\to S_1', \varphi_1'(X)=\exp(X')\cdot x_{w_1'},  
\end{split}
\end{equation}

\begin{Lem}\label{lem:main}  With $w_1,w\in\cW$ and $w_1',w'\in\cW'$ as above and with $Z_{w_1}\subset Z_w$, we have 
$$
Z_w\cap S_1\simeq Z_{w'}\cap S_1'.
$$
\end{Lem}
\begin{proof}
We prove the two cases separately.  \\
\underline{Case I}. In this case we have $\bar{\fn}\cap\bar{\fn}^{w_1}=\bar{\fn}'\cap\bar{\fn}'^{w_1}$, so it suffices to show that for $X\in\bar{\fn}\cap\bar{\fn}^{w_1}$, $\varphi_1(X)\in Z_w$ if and only if $\varphi_1'(X)\in Z_{w'}$.  Note that the flag $x_w$ has first term $\C e_{2n}$, since $s_{2\ge_1}(e_1)=e_{2n}$.

Since $X\in \bar{\fn}'\subset\fg'$,
\begin{equation*}
\begin{split}
  & \varphi_1(X)=(F_i); \; F_i=\C e_{2n}\oplus \exp(X)\left(\spa\{e_{w_1(j)} : j=2,3,\dots,i\}\right)  \\
  & \varphi_1'(X)=(F_i'); \;F_i'=\C e_{1}\oplus \exp(X)\left(\spa\{e_{w_1'(j)} : j=2,3,\dots,i\}\right)   \\
  &\hspace{90pt}=\C e_{1}\oplus \exp(X)\left(\spa\{e_{w_1(j)} : j=2,3,\dots,i\}\right) . 
\end{split}
\end{equation*}
It is clear from this that 
$$\dim(F_i'\cap (\C^n\times \{0\}))=\dim(F_i\cap(\C^n\times\{0\}))+1.
$$
By equation (\ref{eqn:as}), it follows that $(F_i)\in Z_w\iff (F_i')\in Z_{w'}$. \\
\underline{Case II}.  In this case $\bar{\fn}\cap\bar{\fn}^{w_1}=\gs\left(\bar{\fn}'\cap\bar{\fn}'^{w_1'} + \sum_{j}\fg^{(-\ge_1-\ge_j)}\right).$  When $X=\gs(X'), X'\in\bar{\fn}'\cap\bar{\fn}'^{w_1'}$, we have
\begin{align*}
  \varphi_1(X)&=\exp(\gs(X'))\cdot x_w  \\
    & =\exp(\gs(X'))\gs \cdot x_{w'}  \\
    & =\gs\exp(X')\cdot x_{w'}  \\
    & =\gs\varphi_1'(X').
\end{align*}
Now write $\varphi_1(X)=(F_i)$ and $\varphi_1'(X')=(F_i')$, so the above calculation says $(F_i)=(\gs F_i')$.  Then
\begin{align*}
  \varphi_1(X)&\in Z_w  \iff \dim(F_i\cap(\C^n\times\{0\}))\geq a_i  \\
    & \iff \dim(\gs F_i'\cap (\C^n\times\{0\}))\geq a_i  \\
    & \iff \dim(F_i'\cap(\C^n\times\{0\}))\geq a_i, \text{ since }\gs(\C^n\times\{0\})=\C^n\times\{0\},  \\
    & \iff \varphi_1'(X')\in Z_{w'}, \text{ by (\ref{eqn:as})}.
\end{align*}
It will follow that $Z_w\cap S_1=\gs(Z_{w'}\cap S_1')$, once we show that if $X=\gs(X')+\gs(Y), X'\in\bar{\fn}'\cap\bar{\fn}'^{w_1'}$ and $Y\in \sum_{j}\fg^{(-\ge_1-\ge_j)},$ then $\varphi_1(X)\in Z_w$ can only hold when $Y=0$.

We now check this.  Write $\varphi_1(X)=\exp(\gs(X')+\gs(Y))\cdot x_w=(F_i)$.  Then $F_1=\exp(\gs(X')+\gs(Y))\cdot\C e_n=\gs\exp(X'+Y)\cdot\C e_1=\gs\exp(Y)\exp(X')\cdot \C e_1=\gs\exp(Y)\cdot\C e_1.$  Thus, $F_1$ is the span of 
$$
\gs\exp(Y)e_1=\gs \begin{bmatrix}
  1 & & & &  & & \\
  & \ddots &  & & & & \\
   & &1  &  & & & \\
  y_1 &  & 0 & 1 & & & \\
   \vdots& &  & & & \ddots & \\
  y_n &\dots & y_1& & & & 1
\end{bmatrix}
\begin{bmatrix}
   1 \\
   0 & \\
   \vdots & \\
    0& \\
   \vdots& \\
   0& 
\end{bmatrix}=
\begin{bmatrix}
   0 \\
   \vdots & \\
   1 & \\
   y_1 & \\
   \vdots& \\
   y_n& 
\end{bmatrix}.
$$
Since $w=(n\cdots)$, $a_1=1.$  Therefore, $\varphi_1(X)\in Z_w$ implies $\dim(F_1\cap(\C^n\times\{0\}))\geq 1$.  This can only happen when $Y=0$.
\end{proof}

\subsection{}\label{prop:first}
Let $w_1,w$ and $w_1',w'$ be as in the preceding section.  In particular $Z_{w_1}\subset Z_w$ and $Z_{w_1'}\subset Z_{w'}$.  From Lemma \ref{lem:main} we may conclude that $CC(\IC_{Z_w})$ contains $[T_{Z_{w_1}}^*\fB]$ with the same multiplicity that $CC(\IC_{Z_{w'}})$ contains $[T_{Z_{w_1'}}^*\fB']$. 

Since $w\leftrightarrow c$ means $Z_w=\fQ_c$ we may conclude the following proposition.

\begin{Prop}\label{prop:slices-clan}
Let $c,c_1\in\Cl$ and $\fQ_{c_1} \subset\bar{\fQ_c}$.
\begin{enumerate}
\item If $c=(+\hspace{1pt}c'),c_1=(+\hspace{1pt}c_1')$, then $T_{c_1}$ occurs in $CC(X_c)$ with the same multiplicity that $T_{c_1'}$ occurs in $CC(X_{c'})$.
\item  If $c=(1\,c'),c_1=(1\,c_1')$, then $T_{c_1}$ occurs in $CC(X_c)$ with the same multiplicity that $T_{c_1'}$ occurs in $CC(X_{c'})$.
\end{enumerate}
\end{Prop}


\section{Characteristic cycles}  In this section we determine all characteristic cycles of highest weight Harish-Chandra modules.

\subsection{} In \S\ref{subsec:bo} we showed that for two orbits $\fQ_1,\fQ$ whose closures are supports of highest weight Harish-Chandra modules, $\fQ_1\subset\bar{\fQ}$ implies $\mu(\bar{T_{\fQ_1}^*\fB})\supseteq\mu(T_{\fQ}^*\fB)$.  This immediately gives us that
\begin{equation}\label{eqn:lem-bo}
\fQ_i\in\Cg(i)\text{ and }\fQ_j\in\Cg(j)\text{ with }i<j\text{ implies }\fQ_i\nsubseteq\bar{\fQ_j}.
\end{equation}
This leads to an immediate narrowing down of possible terms in characteristic cycles.

\begin{Prop}\label{prop:nd}  Suppose $k$ is even and $0\leq k\leq n$.  Let $X_\fQ\in\Chc(k)$ with $\supp(X_\fQ)$.  Then the following hold.
\begin{enumerate}
\item If $\fQ\in\Cg(k)$ and $\bar{T_{\fQ_1}^*\fB}$ occurs in $CC(X_\fQ)$, then $\fQ_1\in\Cg(k)$.
\item If $\fQ\in\Cg(k-1)$ and $\bar{T_{\fQ_1}^*\fB}$ occurs in $CC(X_\fQ)$, then $\fQ_1\in\Cg(k)\cup\Cg(k-1)$.
\end{enumerate}
When $n$ is odd and $\fQ\in\Cg(n)$, if $\bar{T_{\fQ_1}^*\fB}$ occurs in $CC(X_\fQ)$, then $\fQ_1\in\Cg(n)$.
\end{Prop}

\begin{proof}  Let $k$ be even and suppose $\bar{T_{\fQ_1}^*\fB}$ occurs in $CC(X_\fQ)$.  Then we know that (i) $\mu(\bar{T_{\fQ_1}^*\fB})\subseteq AV(X_\fQ)=\bar{\Cal O}_k$ and (ii) $\fQ_1\subset\bar{\fQ}$.  Suppose $\fQ_1\in\Cg(j)$.  When $\fQ\in\Cg(k)$, (i) says $j\leq k$ and (ii) (along with (\ref{eqn:lem-bo})) gives $j\geq k$.  Now (1) follows. When  $\fQ\in\Cg(k-1)$, (i) and (ii) give $j\leq k$ and $j\geq k-1$.  So (2) is proved.

For the last statement, if $\fQ_1\in \Cg(j)$ and $\fQ_1\subset\bar{\fQ}$,  then by the comment immediately before (\ref{eqn:lem-bo}) we have $\mu(\bar{T_{\fQ_1}^*\fB})\supseteq \bar{\Cal O}_n$, so $\fQ_1\in\Cg(n)$.
\end{proof}


\subsection{} We now determine the characteristic cycles of those $X_\fQ$ with $\fQ\leftrightarrow c=(+\hspace{1pt}c')\in\Cl$.  It is a fact that such $X_\fQ$ are cohomologically induced representations $\Cal R_{\bar{\fq}}(Z)$, where $\bar{\fq}=\fl+\bar{\fu}$.  It is a general fact that characteristic cycles behave well under cohomological induction. We will rely on our Proposition \ref{prop:slices-clan} directly.

\begin{Prop}\label{prop:cc+}
Let $\fQ$ have clan $c=(+\hspace{1pt}c')\in\Cl$ and let $X_\fQ$ be the highest weight Harish-Chandra module with support $\bar{\fQ}$.  Suppose $\fQ'$ corresponds to $c'\in\Cl'$ and
$$CC(X_{\fQ'})=\sum_{c_1'}m_{c_1'}[T_{c_1'}].
$$ 
Let $\fQ_1'$ be the $K'$-orbit corresponding to $c_1'$, and $\fQ_1$ corresponding to $c_1=(+\hspace{1pt}c_1')$.  Then
\begin{equation}\label{eqn:cc+}
CC(X_\fQ)=\sum_{c_1'} m_{c_1'}[T_{c_1}].
\end{equation}
\end{Prop}

\begin{proof}  Since the clan of $\fQ$ is $c=(+\hspace{1pt}c')$, no $\fQ_1$ with clan beginning with $1$ can be in the closure of $\fQ$.  So, by point (b) in Section \ref{subsec:1-cc}, $\fQ_1$ has clan of the form $c_1=(+\hspace{1pt}c_1')$.  Now Proposition \ref{prop:slices-clan}(1) says that for such $\fQ_1$, $\bar{T_{\fQ_1}^*\fB}$ occurs in $CC(X_\fQ)$ if and only if $\bar{T_{\fQ_1'}^*\fB}$ occurs in $CC(X_{\fQ'})$ (and with the same multiplicity).
\end{proof}

Now consider the cell $\Chc(n)$ when $n$ is odd.  By Proposition \ref{prop:34}, Proposition \ref{prop:size} and equation (\ref{eqn:sp-dim}) we conclude that $\Chc(n)=\Cg(n)$ and the cell representation is irreducible.  It follows from \cite{Joseph80b} and \cite{BarbaschVogan82} that the irreducible representations in $\Chc(n)$ have distinct annihilators.  From equation (\ref{eqn:g-cells}) each irreducible representation in $\Chc(n)$ has support parametrized by a clan of the form $(+\,c')$, so the characteristic cycle of $X_c$ is determined by Proposition \ref{prop:cc+}.  By Proposition \ref{prop:nd}, each irreducible representation in $\Chc(n)$ has leading term cycle equal to its characteristic cycle.  These facts are summarized in the following proposition.

\begin{Prop}\label{prop:ccnodd}
If $n$ is odd, then $\Chc(n)=\Cg(n)$ and the corresponding cell representation is irreducible.  If $X_c$ is an irreducible representation in $\Chc(n)$, then the support $\bar{\fQ}_c$ has clan of the form $c=(+\,c')$ and $LTC(X_c)=CC(X_c)$.  The irreducible representations in $\Chc(n)$ have distinct annihilators.
\end{Prop}

\subsection{}\label{subsec:ld} In this section we give the characteristic cycles for the low rank cases.  For $n=1$ ($SL(2)$), the clans are $(+)$ and $(1)$; both characteristic cycles are just the conormal bundle closure of the support (since the closures of the supports are smooth).

When $n=2$, (\ref{eqn:g-cells}) and Proposition \ref{prop:hc-cells} give the cells as listed on the following table.  To verify the formulas for the characteristic cycles, let's first look at $(++),(+1)$ and $(12)$.  The closure criterion and the $\tau$-invariant criterion imply that $CC(X_c)$ for these clans is just the conormal bundle closure to $\fQ_c$.  This also follows from the fact that the three $K$-orbit closures with clans $(++),(+1)$ and $(12)$ are smooth.  Now consider $c=(1+)$.  The moment map image is too small to be the associated variety, so there must be some $T_{c'}$, $c'\neq c$, occurring in the characteristic cycle.  The $\tau$-invariant criterion excludes all but $c'=(++)$.  Therefore, $CC(X_{(1+)})=[T_{(1+)}]+m[T_{(++)}]$, for some $m\geq 1$.  

To establish that $m=1$, we consider coherent continuation.  Let $s_2$ be the simple reflection for the long simple root.  Then in the coherent continuation
\begin{equation*}
s_2\cdot X_{(1+)} =X_{(1+)}+X_{(+1)}\text{ and }
s_2\cdot X_{(++)} =X_{(++)}+X_{(+1)}.
\end{equation*}
This formula may be found, for example, using the ATLAS software.  Using the $W$-equivariance of the $CC$ map (\cite{Tanisaki85})
\begin{align*}
s_2&\cdot [T_{(++)}]=s_2\cdot CC(X_{(++)})=CC(s_2\cdot X_{(++)})  \\
&=CC(X_{(++)}+X_{(+1)})=[T_{(++)}]+[T_{(+1)}].
\end{align*}
This gives
\begin{align*}
s_2&\cdot[T_{(1+)}]+m\left([T_{(++)}]+[T_{(+1)}]\right)  \\
&=s_2\cdot[T_{(1+)}]+m\hspace{1pt}s_2\cdot[T_{(++)}]  \\
&=s_2\cdot CC(X_{(1+)})  \\
&=CC(s_2\cdot X_{(1+)})  \\
&=\left( [T_{(1+)}]+m\hspace{1pt}[T_{(++)}]\right)+[T_{(+1)}]
\end{align*}
Therefore,
\begin{equation*}
s_2\cdot[T_{(++)}]=[T_{(++)}]+(1-m)[T_{(+1)}].
\end{equation*}
However, the action of $s_2$ has nonnegative coefficients ($\ga_2$ is not in the $\tau$-invariant), so $m\leq 1$.  Therefore the multiplicity $m$ is one.

Alternatively, one may compute a normal slice to get a cone in $\C^3$.  Thus the multiplicity is the same as that of $\{0\}$ in the nilpotent cone in $\f{sl}(2)$.  This multiplicity is know to be $1$ (\cite[\S 3]{EvensMirkovic99}).

\begin{center}   
\begin{tabular}{|l|c|c|c|c|c|c|c|c|c|c|c|}
\hline
    & clan & $\dim(\fQ)$ & tau  & $\Chc$ & $\Cg$ &  &  CC  \\
\hline
\hline 
1.  & ++   & 1  & 1   & 2 &  2  &  & $T_{(++)}$    \\
2.  & +1   & 2  & 2   & 2 &  2  &  & $T_{(+1)}$   \\
3.  & 1+   & 3  & 1   & 2 &  1  &  & $T_{(1+)}+T_{(++)}$  \\
\hline

4.  & 1\,2 & 4  & 1,2 & 0 &  0  &  & $T_{(1\,2)}$      \\
\hline
\end{tabular}

{$Sp(4,\R)$}
\end{center}

Now consider $n=3$.  Proposition \ref{prop:cc+} gives the characteristic cycles for clans of the form $c=(+\hspace{1pt}c)$.  If we can exclude conormals of the form $T_{(+\hspace{1pt}c_1')}$ from occurring in $CC(X_{(1,c')})$, then Proposition \ref{prop:slices-clan}(2) will determine all $CC_{(1\,c')}$.  This gives the characteristic cycles on the following table.

\begin{center}
\begin{tabular}{|l|c|c|c|c|c|c|c|c|c|c|c|}
\hline
     & clan & $\dim(\fQ)$ & tau  &$\Chc$ & $\Cg$ &    CC  \\
\hline
\hline
1.  &  +++ & 3 & 1,2 & 3 & 3 &   $T_{(+++)}$    \\
2.  &  ++1 & 4 & 1,3 & 3 & 3 &   $T_{(++1)}$      \\
3.  &  +1+ & 5 & 2 & 3 & 3 &   $T_{(+1+)}+T_{(+++)}$      \\
\hline
4.  &  1++ & 6 & 1,2 & 2 & 2 &   $T_{(1++)}$      \\
5.  &  +1\,2 & 6 & 2,3 & 2 & 2   &  $T_{(+1\,2)}$  \\
6.  &  1+2 & 7   & 1,3 & 2 & 2   & $T_{(1+2)}$      \\
7.   & 1\,2+ & 8 & 1,2 & 2 & 1 &    $T_{(1\,2+)}+T_{(1++)}$      \\
\hline
8.  & 1\,2\,3& 9& 1,2,3 & 0 & 0 &   $T_{(1\,2\,3)}$ \\
\hline
\end{tabular}

{$Sp(6,\R)$}
\end{center}

Let us argue that $T_{(+\hspace{1pt}c_1')}$ cannot occur in $CC(X_{(1,c')})$.  We know that $CC(X_{(123)})=[T_{(123)}]$, since $(123)\leftrightarrow \bar{\fQ}=\fB$ (smooth).  The other three clans of the form $c=(1\,c')$ are in $\Chc(2)$, so any $T_{c_1}$ in $CC(X_{c})$ must have $c_1\in\Cg(1)\cup\Cg(2)$.  The only possibility of the form $c_1=(+\hspace{1pt}c_1')$ is $c_1=(+12)$.  However, $\ga_1\in\tau(c_1)\,\smallsetminus\, \tau(c)$, showing that $T_{c_1}$ cannot occur in $CC(X_{(1\,c')})$.  For arbitrary $n$ we will use Lemma \ref{lem:54-A} to exclude $T_{(+\,c_1')}$ from occurring in $CC(X_{(1\,c'))}$ for all cells except $\Chc(n)$, $n$ even.  This is contained in \S\ref{subsec:CC}.

When $n=4$ there is a new phenomenon.  For the cells $\Chc(0)$ and $\Chc(2)$ the argument for the $CC$ is exactly as in the $n=3$ case.  For $\Chc(n), n=4,$ $T_{(+\hspace{1pt}c_1')}$ cannot be excluded from $CC(X_{(1\,c')})$.  In fact $T_{(+\hspace{1pt}c_1')}$ and $T_{(1\,c_1')}$ are exactly the terms appearing in $CC(X_c)$ as $c_1'$ runs over all clans for which $T_{c_1'}$ occurs in $CC(X_{c'})$ for the smaller group $G'$.  We deal with $\Chc(n)$ ($n$ even) in \S\ref{subsec:harder}.  The following table of characteristic cycles will then be justified.

\begin{center}
\begin{tabular}{|l|c|c|c|c|c|c|c|c|c|c|c|}
\hline
    &  clan & $\dim(\fQ)$ & tau  & $\Chc$ & $\Cg$ &    CC \\
\hline\hline
1.    &  ++++ & 6 & 1,2,3 & 4 & 4 &   $T_{(++++)}$  \\
2.    & +++\,1 & 7 & 1,2,4 & 4 & 4   & $T_{(+++\,1)}$      \\
3.    & ++\,1+ & 8 & 1,3 & 4 & 4 &   $T_{(++\,1+)}+T_{(++++)}$        \\
4.    & +\,1++ & 9 & 2,3 & 4 & 4 &   $T_{(+\,1++)}$   \\

5.    & ++\,1\,2 & 9 & 1,3,4 & 4 & 4   & $T_{(++\,1\,2)}$      \\
6.    & +\,1+\,2 & 10 & 2,4 & 4 & 4 &   $T_{(+\,1+\,2)}$     \\
7.    & \,1+++ & 10 & 1,2,3 & 4 & 3 &   $T_{(\,1+++)}+T_{(++++)}$  \\
8.    & \,1++\,2 & 11 & 1,2,4 & 4 & 3   & $T_{(\,1++\,2)}+T_{(+++\,1)}$  \\

9.    & +\,1\,2+ & 11 & 2,3 & 4 & 3 &   $T_{(+\,1\,2+)}+T_{(+\,1++)}$     \\
10.   & \,1+\,2+ & 12 & 1,3 & 4 & 3 &   $T_{(\,1+\,2+)}+T_{(\,1+++)}+T_{(++\,1+)}+T_{(++++)}$    \\
\hline
11.   & +\,1\,2\,3 & 12 & 2,3,4 & 2 & 2   & $T_{(+\,1\,2\,3)}$    \\
12.   & \,1+\,2\,3 & 13 & 1,3,4 & 2 & 2   & $T_{(\,1+\,2\,3)}$    \\
13.   & \,1\,2++ & 13 & 1,2,3 & 2 & 2 &   $T_{(\,1\,2++)}$  \\
14.   &  \,1\,2+\,3 & 14 & 1,2,4 & 2 & 2   & $T_{(\,1\,2+\,3)}$     \\
15.   & \,1\,2\,3+ & 15 & 1,2,3 & 2 & 1 &   $T_{(\,1\,2\,3+)}+T_{(\,1\,2++)}$   \\
\hline
16.   & \,1\,2\,3\,4 & 16 & 1,2,3,4 & 0 & 0   & $T_{(\,1\,2\,3\,4)}$      \\
\hline
\end{tabular} \\
$Sp(8,\R)$
\end{center}

\subsection{}\label{subsec:CC} We now turn to $\fQ$ with clan of the form $c=(1\,c')$ and we  determine $CC(X_c)$ when $c\in\Chc(k)$ with $k$ even and $k\neq n$.  Note that in the case of $k=n$, $n$ odd, is contained in Proposition \ref{prop:ccnodd}.  The case of $\Chc(n)$, $n$ even, is left for \S\ref{subsec:harder}.

\begin{Prop}\label{prop:cc1}
  Suppose $k$ is even and $k\neq n$.  Let $c=(1\,c')\in\Chc(k)$.  If $CC(X_{c'})=\sum m_{c_1'}[T_{c_1'}]$, then 
\begin{equation}\label{eqn:cc1}
  CC(X_{(1\,c')})=\sum m_{c_1'}[T_{(1\,c_1')}].
  \end{equation}
\end{Prop}

Proposition \ref{prop:slices-clan}(2) says that all terms on the right-hand side of (\ref{eqn:cc1}) occur with the stated multiplicities.  The lemma also says that no other $T_{(1\,c_1')}$ can occur.

The following lemma completes the proof of the proposition by excluding any term $T_{(+\,c_1')}$ from occurring in $CC(X_c)$.

\begin{Lem}\label{lem:54-A}
Suppose $k$ is even and $k\neq n$.  If $c=(1\,c')\in\Chc(k)$ and $T_{c_1}$ occurs in $CC(X_c)$, then $c_1$ is of the form $(1\,c_1')$.
\end{Lem}
\begin{proof} Suppose that $c_1=(+\hspace{1pt}c_1')$ occurs in $CC(X_c), c\in\Chc(k)$.  We use induction on $n$ to obtain a contradiction.  The cases $n\leq 3$ have been established in \S\ref{subsec:ld}.  Assume $n>3$.  By Proposition \ref{prop:nd}, $c_1\in\Chc(k)$.

\noindent{\bf Step 1.} If $c=(1+2\dots)$ and $c_1=(++1\dots)$, then $T_{c_1}$ does not occur in $CC(X_c)$.  To prove this we assume otherwise.  For $\ga=\ga_2,\gb=\ga_1$, both $c$ and $c_1$ are in the domain of $\rm{\mathbf {T}}_{\ga\gb}=\rm{\mathbf {T}}_{21}$.  By Lemma \ref{lem:mcgovern}
\begin{equation}\label{eqn:55-1}
  T_{\rm{\mathbf {T}}_{\ga\gb}(c_1)} \text{ occurs in }CC(X_{\rm{\mathbf {T}}_{\ga\gb}(c)}).
\end{equation}
By (\ref{eqn:tabc})
\begin{align*}
  &\rm{\mathbf {T}}_{\ga\gb}(c)\,=\,(+1\;2\;\dots) \\
  &\rm{\mathbf {T}}_{\ga\gb}(c_1)=(+1+\dots).
\end{align*}  
Proposition \ref{prop:slices-clan}(2), (\ref{eqn:55-1}) tells us that 
\begin{align*}
  T_{(1+\dots)} \text{ occurs in } CC(X_{(12\dots)}), \text{ for }G'=Sp(2(n-1)).  
\intertext{Proposition \ref{prop:slices-clan}(1) implies}
  T_{(+\dots)} \text{ occurs in } CC(X_{(1\dots)}), \text{ for }G''=Sp(2(n-2)).  
\end{align*}
Since $c=(1+2\dots)\in\Chc(k)$, we have $(+\,1\dots)\in\Chc'(k)$.  Therefore, $(1\dots)\in\Chc''(k-2)$.  Since $k-2\neq n-2$, our inductive hypothesis says that no $T_{(+\dots)}$ can occur in $CC(X_{(1\dots)})$ (for $G''=Sp(2(n-2))$), giving a contradiction.

\noindent{\bf Step 2.} Now consider arbitrary $c=(1\,c')$ and $c_1=(+\hspace{1pt}c_1')$.  Then $c_1=(+\dots +1\dots)$ with $r\geq 1$ consecutive $+$'s at the beginning.  Note that if no `$1$' occurred, then we would be in $\Chc(n)$, which we have assumed from the beginning is not the case.  Since $\tau(c_1)\supseteq\tau(c)$, $c$ must have a `$+$' in the $r^{\text{th}}$ place:
\begin{align*}
  &c\,=(\,1\dots \;+m\dots)  \\
  &c_1=(+\dots +1\dots).
\end{align*}
There are two possibilities:
\begin{align}
\begin{split}\label{eqn:55-3}
  &c\,=(\,1\,\dots \,++\,m\dots)  \\
  &c_1=(+\dots ++\,\;1\dots)
\end{split}
\intertext{and}
\begin{split}\label{eqn:55-4}
  &c\,=(1\dots m\text{-}1+m\dots)  \\
  &c_1=(+\dots +\;+\;\,1\,\dots).
\end{split}
\end{align}
Let $\ga=\ga_{r},\gb=\ga_{r-1}.$  Then in both cases, both $c$ and $c_1$ are in the domain of $\rm{\mathbf {T}}_{\ga\gb}$ and in the first case we have
\begin{align}
\begin{split}\label{eqn:55-5}
  &\rm{\mathbf {T}}_{\ga\gb}(c)\,=(\,1\,\dots \,+\,m+\dots)  \\
  &\rm{\mathbf {T}}_{\ga\gb}(c_1)=(+\dots +\;1+\dots)
\end{split}
\intertext{and in the second case we have}
\begin{split}\label{eqn:55-6}
  &\rm{\mathbf {T}}_{\ga\gb}(c)\,=(1\dots +m\text{-}1\;m\dots)  \\
  &\rm{\mathbf {T}}_{\ga\gb}(c_1)=(+\dots +\;\,1\,+\dots).
\end{split}
\end{align}
Continuing we reach the case of 
\begin{align*}
  &(\,1+2\dots)  \\
  &(++1\dots).
\end{align*}
This gives that $T_{(++1\dots)}$ occurs in $CC(X_{(1+2\dots)})$ (Lemma \ref{lem:mcgovern}).  Since $\rm{\mathbf {T}}_{\ga\gb}$ preserves associated variety, so preserves Harish-Chandra cells, $(1+2\dots)$ is in $\Chc(k)$, so step 1 applies, giving a contradiction.
\end{proof}


\subsection{}\label{subsec:harder} The characteristic cycle of any irreducible Harish-Chandra module $X\in\Chc(n)$ with $n$ even is determined in this section.  This will complete our description of characteristic cycles for highest weight Harish-Chandra modules for $G_\R=Sp(2n,\R)$ in terms of those of highest weight Harish-Chandra modules for $G_\R'=Sp(2(n-1),\R)$.  If $\supp(X)=\bar{\fQ}_{(+\hspace{1pt}c')}$, then we know the characteristic cycle by Proposition \ref{prop:cc+}.  If $\supp(X)=\bar{\fQ}_{(1\,c')}$ with $c=(1\,c')\notin \Chc(n), n$ even, then Proposition \ref{prop:cc1} gives the characteristic cycle.  We address the only remaining case: $\supp(X)=\bar{\fQ}_{(1\,c')}$, $c=(1\,c')\in \Chc(n), n$ even.  The answer is different than when $X$ is in other cells, as the following proposition shows. 

\begin{Prop}\label{prop:top-cell} Let $X_{(1\,c')}\in \Chc(n), n$ even.  Suppose that 
$$ CC(X_{c'})=\sum m_{c_1'}[T_{c_1'}].
$$
Then 
\begin{equation}\label{eqn:harder}
 CC(X_{(1\,c')})=\sum m_{c_1'}[T_{(1\,c_1')}] + \sum m_{c_1'}[T_{(+\hspace{1pt}c_1')}].
\end{equation} 
The second term on the right-hand side is the leading term cycle.
\end{Prop}

The remainder of this section is devoted to the proof of this proposition.  We assume that $n$ is even and $c=(1\,c')\in\Chc(n)$. 

\begin{Lem}\label{lem:55-A}  If $T_{(+\hspace{1pt}c_1')}$ occurs in $CC(X_{(1\,c')})$, then $c_1'\in\Cg'(n-1)$.
\end{Lem}

\begin{proof} By Proposition \ref{prop:hc-cells} and (\ref{eqn:g-cells}), we must show that $T_{c_1}$ cannot occur when $c_1=(+\hspace{1pt}c_1')\in\Cg(n-2)\cup\Cg(n-3)$.  This was nearly done in the proof of Lemma \ref{lem:54-A}.  The same arguments given in Step 1 of the proof of Lemma \ref{lem:54-A} apply to give the following.  If
\begin{align*}
  c&=(1+2\,\tilde{c})=(1+2\dots)  \\
  c_1&=(++1\,\tilde{c}_1)=(++1\dots)
\end{align*}
and $T_{c_1}$ occurs in $CC(X_c)$, then 
$$T_{(+\tilde{c}_1)} \text{ occurs in } CC(X_{(1\tilde{c})}), \text{ in }G''
$$
and $(+\tilde{c}_1)$ and $(1\tilde{c})$ are in $\Chc''(n-2)$.  When $c_1'\in \Cg'(n-2)\cup\Cg'(n-3)$, $\tilde{c}_1\in\Cg'(n-4)\cup\Cg'(n-5)$.  Induction says this cannot happen.

Step 2 of the proof of Lemma \ref{lem:54-A} applies here.
\end{proof}

Continue with $c=(1\,c')\in\Chc(n)$.  Then (by (\ref{eqn:g-cells})) $c\in\Cg(n-1)$, so $c'\in\Cg'(n-1)$.  It follows that $c'=(+\hspace{1pt}c'')$, for some $c''\in\Cg''(n-2)\cup\Cg''(n-3)$.

\begin{Lem}\label{lem:55-B}
If $c=(1\,c')\in\Chc(n)$, then $T_{(+\hspace{1pt}c')}$ occurs in $LTC(X_c)$ with multiplicity one.
\end{Lem}
\begin{proof} We proceed by induction on $m,n=2m$  If $m=1$ this is contained in \S\ref{subsec:ld}.  Assume $m\geq 2$ and the statement holds for $n-2=2(m-1)$.

\noindent {\bf Case 1}.  First consider $c'=(+\hspace{1pt}c'')=(+1\dots)$.  Then by the comments just before the statement of this lemma and (\ref{eqn:g-cells}), $c''\in\Cg''(n-3)$.  Write $c_1=(+\hspace{1pt}c')$.  Both $c_1$ and $c$ are in the domain of $\rm{\mathbf {T}}_{21}$ and by (\ref{eqn:tabc})
\begin{align*}
  \rm{\mathbf {T}}_{21}(c_1)&=\rm{\mathbf {T}}_{21}(++1\dots)=(+1+\dots)  \\
  \rm{\mathbf {T}}_{21}(c)&=\rm{\mathbf {T}}_{21}(1+2\dots)=(+1\,2\dots).
\end{align*}
Therefore, 
\begin{align*}
 T_{c_1}& \text{ occurs in } CC(X_c)  \\
 &\iff   T_{(+1+\dots)} \text{ occurs in } CC(X_{(+12\dots)}), \text{ by \ref{lem:mcgovern},}  \\
 &\iff   T_{(1+\dots)} \text{ occurs in } CC(X_{(12\dots)}), \text{ by Proposition \ref{prop:cc+} (for $(G',K')$),}  \\
 &\iff   T_{(+\dots)} \text{ occurs in } CC(X_{(1\dots)})  \text{ by Proposition \ref{prop:cc1} (for $(G'',K'')$).}
\end{align*}
The multiplicity is preserved under each equivalence (by Lemma \ref{lem:mcgovern+} for the first).  By induction, the last line holds and the multiplicity is one.

\noindent {\bf Case 2}.  Now consider any $c'\in\Cg'(n-1)$ different from $(++\dots +)$.   We reduce to the above case of $c'=(+1\dots)$.  We know that $c'$ begins with $+$.  Let $r$ be the number of leading $+$'s in $c'$.  Then 
\begin{align*}
  c_1&=(++\dots +1\dots)=(++\hspace{1pt}c'') \\
  c&=(1+\, \dots \,+2\dots)=(1\hspace{1pt}+\hspace{1pt}c'')
\end{align*}
are both in the domain of $\rm{\mathbf {T}}_{r+1\,r}.$  Applying $\rm{\mathbf {T}}_{r+1\,r}$ moves the $1$ (resp., $2$) one slot to the left in $c_1$ (resp., $c$).  Continuing one gets
\begin{align*}
  \rm{\mathbf {T}}_{32}\rm{\mathbf {T}}_{43}\dots \rm{\mathbf {T}}_{r+1\,r}(c_1)&=(++1\dots) \\
  \rm{\mathbf {T}}_{32}\rm{\mathbf {T}}_{43}\dots \rm{\mathbf {T}}_{r+1\,r}(c)&=(1\,+2\dots).
\end{align*}
We conclude that 
$T_{c_1}$ occurs in $CC(X_c)$.  The multiplicity is one as each $\rm{\mathbf {T}}_{j+1\,j}$ preserves multiplicity by Lemma \ref{lem:mcgovern+}.

\noindent {\bf Case 3}.  Now let $c=(1\,c')=(1+\dots++)$ and $(+\,c')=(++\dots+)$.  We show that $CC(X_c)=[T_{(1+\dots+)}]+[T_{(++\dots+)}]$.  The first step is to show that $CC(X_c)=[T_{(1+\dots+)}]+m[T_{(++\dots+)}]$, for some $m\geq 1$.  Then we show that $m=1$ by determining that 
$$s_n\cdot T_{(1+\dots+)}= T_{(1+\dots+)}+(1-m) T_{(+\dots+1)}+\mathrm{(other\, terms)}.
$$
As $\ga_n\notin\tau(1+\dots+), s_n\cdot  T_{(1+\dots+)}$ is a linear combination of conormal bundles with nonnegative coefficients \cite[\S3]{Tanisaki85}, so $m$ must be $1$.

Consider $(1+\dots+2)$ and $(+\dots+1)$.  These are in the domain of $\rm{\mathbf {T}}_{n-1\,n}$.  By Case (2) above,  $T_{(+\dots+1)}$ occurs in $CC(X_{(1+\dots +2)}$.  By (\ref{eqn:tabcc}), 
\begin{align*}&\rm{\mathbf {T}}_{n-1\,n}(1+\dots+2)=\{(1+\dots+2+),(1+\dots+)\} \text{ and } \\ &\rm{\mathbf {T}}_{n-1\,n}(++\dots+1)=\{(+\dots+1),(++\dots+)\}.
\end{align*}
Lemma \ref{lem:mcgovern} tells us that one of $ T_{(++\dots+)}$ and $T_{(+\dots+1+)}$ occurs in $CC(X_{(1+\dots+)})$.  The $\tau$-invariant excludes $T_{(+\dots+1+)}$ from occurring.  Therefore
$$CC(X_c)=[T_c]+m[T_{(++\dots+)}]+\mathrm{(other\,terms)}.
$$ 
The other terms each have $\tau$-invariant containing $\tau(c)=\{\ga_1\ga_2,\dots,\ga_{n-1}\}$.  The only possibilities are $T_{(12\dots k++\dots+)}$.  These are excluded (for $k>1$) by Proposition \ref{prop:nd}.  Therefore 
$$CC(X_c)=[T_c]+m[T_{(++\dots+)}].$$

By the $W$-equivariance of the characteristic cycle map (\cite{Tanisaki85}) $s_n\cdot CC(X_c)=CC(s_n\cdot X_c)$.  We compute both sides and compare.
\begin{align}\notag
s_n\cdot &CC( X_{(1+\dots+)})=s_n\cdot [T_{(1+\dots+)}]+ms_n\cdot [T_{(++\dots+)}] \\
\notag&=s_n\cdot[T_{(1+\dots+)}]+m\hspace{1pt} s_n\cdot CC(X_{(++\dots+)}), \text{ since $\fQ_{(++\dots+)}$ is smooth}, \\
\notag&=s_n\cdot[T_{(1+\dots+)}]+m\hspace{1pt} CC(s_n\cdot X_{(++\dots+)}) \\
\notag&=s_n\cdot[T_{(1+\dots+)}]+m\hspace{1pt} CC(X_{(++\dots+)}+X_{(+\dots+1)}+\sum X_d), \text{ by \ref{eqn:coherent-c-c}}, \ga_{n-1},\ga_n\in\tau(d), \\
\notag&= s_n\cdot[T_{(1+\dots+)}]+m\hspace{1pt} [T_{(++\dots+)}]+m CC(X_{(+\dots+1)})+m\hspace{1pt} \sum CC(X_d) \\
\label{eqn:1}&= s_n\cdot[T_{(1+\dots+)}]+m\hspace{1pt} [T_{(++\dots+)}]+   m[T_{(+\dots+1)}]+m\hspace{1pt} \sum CC(X_d),\\
\notag &\quad  \text{ since $\fQ_{(+\dots+1)}$ is smooth}.
\end{align}
On the other hand, this is equal to 
\begin{align}\notag
CC(&s_n\cdot X_{(1+\dots+)})=CC(X_{(1+\dots+)}+X_{(1+\dots+2)}+\sum X_e),\text{ with }\ga_{n-1},\ga_n\in\tau(X_e), \\
\notag &=([T_{(1+\dots+)}]+m[T_{(++\dots+)}])+CC(X_{(1+\dots+2)})+\sum CC(X_e) \\
\label{eqn:2}&=([T_{(++\dots+)}]+m[T_{(++\dots+)}])+([T_{(1+\dots+2)}]+[T_{(+\dots+1)}]  \\
\notag&\qquad+\text{(other conormals)})+\sum CC(X_e), \text{ by Case 2 above}.
\end{align}  
In the last expression `other conormals' means a linear combination of conormal bundles other than $T_{(++\dots +1)}$ and $T_{(++\dots++)}$.

Combining (\ref{eqn:1}) and (\ref{eqn:2}) gives
$$s_n\cdot[T_{(1+\dots+)}]=[T_{(1+\dots+)}]+(1-m)[T_{(+\dots+1)}]+\sum CC(X_e)-\sum CC(X_d)+\text{(others)}.$$
Now observe that $T_{(1+\dots+)}$ does not occur in any $CC(X_d), CC(X_e)$ since $\ga_{n-1}\notin\tau(+\dots+1)$, but $\ga_{n-1}\in\tau(d),\tau(e)$.  We conclude that $$s_n\cdot T_{(1+\dots+)}=[T_{(1+\dots+)}]+(1-m)[T_{(1+\dots+)}]+\text{(others conormal bundles)}.$$
Thus $m=1$.
\end{proof}

Our next lemma relies on some general facts about the Goldie rank polynomial of $\ann(X)$ and the leading term cycle of $X$.  Consider for a moment an arbitrary Harish-Chandra module $X$ for which $AV(X)=\bar{\Cal O}_K$, for some nilpotent $K$-orbit $\Cal O_K=K\cdot f$ in $\fp$.  Write $LTC(X)=\sum m_j [\bar{T_{\fQ_j}^*\fB}]$.  Then $\mu(\bar{T_{\fQ_j}^*\fB})=\bar{\Cal O}_K$, for each $j$.  Since $f\in\Cal N\cap\fp$, $\mu^{-1}(f)\subset \cup \bar{T_{\fQ}^*\fB}$, with the union over all $K$-orbits in $\fB$.  Therefore, $C_\fQ:=\mu^{-1}(f)\cap T_\fQ$ is a union of several irreducible components of the Springer fiber $\mu^{-1}(f)$.  Thus, to each conormal bundle for which $\mu(\bar{T_{\fQ_j}^*\fB})=\bar{\Cal O}_K$ there is a `fiber polynomial'
\begin{equation*}
  q_j(\gl):=l. t. \,\dim(H^0(C_{\fQ_j},\Cal O(\gl))),
\end{equation*}
where $l.t.$ denotes the homogeneous part of highest degree.  By \cite{Chang93}, there is a constant so that the Goldie rank polynomial $P(\gl)$ of $\ann(X)$ is   
\begin{equation}\label{eqn:grp-ltc}
P(\gl)=(const.)\sum m_jq_j(\gl).
\end{equation}
If the component groups $A_G(f)$ and $A_K(f)$ have the same orbits on $\mu^{-1}(f)$, then the set of fiber polynomials is independent.

It is a fact that $A_G(f)$ and $A_K(f)$ have the same orbits on $\mu^{-1}(f)$ for our pair $(G,K)=(Sp(2n),GL(n))$.  See \cite{Trapa07}.

Let us return to the situation of highest weight Harish-Chandra modules for the symplectic group $Sp(2n,\R)$ with $n$ even.  List the clans in $\Cg'(n-1)$ as  $c_1',c_2',\dots, c_l'$, ordered so that $\fQ_{c_i'}\subset  \bar{\fQ}_{c_j'}$ implies $i<j$.  Let

\indent\hbox {
    \parbox{5.2in}{
      $P_j^+(\gl)$ be the Goldie rank polynomial of $\ann(X_{(+\;c_j')}), j=1,\dots,l$ and \\
$q_j(\gl)$ be the fiber polynomial associated to $\fQ_{(+;c_j')}$.}}

\noindent By Proposition \ref{prop:ccnodd}, we know that $CC(X_{c_j'})=LTC(X_{c_j'})$, all $j$.  It follows from Proposition \ref{prop:cc+} that $CC(X_{(+\;c_j')})=LTC(X_{(+\;c_j')})$, all $j$.  Therefore, $P_1^+(\gl),\dots,P_l^+(\gl)$ are independent, since $q_j(\gl)$ occurs in $P_j^+(\gl)$ and does not occur in any $P_i^+(\gl)$ for $i<j$ (given our choice of ordering).  

We also observe that if we write $d_1',d_2',\dots$ for the clans in $\Cg'(n-2)$, then the $(+\;d_i')$ are also in $\Cg(n)$ and, along with the $(+\;c_j')$, make up all of $\Cg(n)$.  By Propositions \ref{prop:cc+} and  \ref{prop:nd}, $CC(X_{(+\;d_i')})$ contains only conormal bundles of the form $T_{(+\;d_k')}$.  Therefore, no $\ann(X_{(+\;d_i')})$ has Goldie rank polynomial containing any $q_j(\gl)$.  Furthermore, the conormal bundle of support is in the leading term cycle for each of $X_{(+\;c_j')}$ and $X_{(+\;d_i')}$, so the Goldie rank polynomials, hence the annihilators, are pairwise distinct.  Since $\#\Cg'(n-1)+\#\Cg'(n-2)=\dim(\pi(\Cal O_n^\C))$ (by Proposition \ref{prop:34} and equation (\ref{eqn:sp-dim})) the  annihilators of the $X_{(+\;c_j')}$ and $X_{(+\;d_i')}$ are all of the annihilators having associated variety $\bar{\Cal O}_n$.  We conclude that if the annihilator of a Harish-Chandra module in $\Chc(n)$ has Goldie rank polynomial containing some $q_j(\gl)$, then this Goldie rank polynomial must be one of the polynomials $P_j^+(\gl)$.

Now consider $\ann(X_{(1\;c_i')})$.  By Lemmas \ref{lem:55-A} and \ref{lem:55-B} we know that 
$$ LTC(X_{(1\;c_j)}) = \sum m_{ij} [T_{(+\;c_i')}].
$$ 
Therefore, Goldie rank polynomial of $\ann(X_{(1\;c_j)})$ is 
\begin{equation*}
  P_j(\gl)=C_j\sum_i m_{ij}q_i(\gl),
\end{equation*}
for some constant $C_j$. Since, by Lemma \ref{lem:55-B}, $P_j(\gl)$ contains $q_j(\gl)$,  we conclude from the previous paragraph that each $P_j(\gl)$ is  $P_k^+(\gl)$ for some $k=1,2,\dots,l$.

\begin{Lem}\label{lem:55-D}
For each $c'\in\Cg'(n-1)$, $\ann(X_{(1\,c')})=\ann(X_{(+\hspace{1pt}c')})$.
\end{Lem}
\begin{proof} First observe that no two of $\ann(X_{(1\;c_j')}), j=1,\dots,l$ are equal.  This is a consequence of \cite{McGovern98} as follows. Since the cell representation $V(\Chc(n))$ has two irreducible constituents, an annihilator of a Harish-Chandra module in $\Chc(n)$ occurs for at most two Harish-Chandra modules in $\Chc(n)$ by \cite[Cor.~3]{McGovern98}.  But we showed above, that the annihilators of $X_{(+\,c_j')}$ and $X_{(+\,d_i')}$ are pairwise distinct and give \emph{all} annihilators from $\Chc(n)$.  Therefore no annihilator from $\Chc(n)$ can occur more than once among the $\ann(X_{(1\,c_j')}), j=1,\dots,l$.

Recall from \cite{Joseph80b}, that two annihilators are the same if and only if their Goldie rank polynomials are the same.
Therefore, it suffices to show that $P_j(\gl)=P_j^+(\gl)$, for each $j$.

We show that $P_k(\gl)=P_k^+(\gl)$, $k=1,2,\dots,l$ using downward induction on $k$.  First suppose that $k=l$.  Then $P_l(\gl)=\sum m_{il}q_i(\gl), m_{ll}\neq 0$, by Lemma \ref{lem:55-B}.  No $P_j^+(\gl)$ contains $q_l(\gl)$ for $j<l$, as a result of the ordering of the $c_j'$.   So $P_l(\gl)=P_l^+(\gl)$.

Now assume $P_{k+1}(\gl)=P_{k+1}^+(\gl),P_{k+2}(\gl)=P_{k+2}^+(\gl),\dots$.  Since $P_k(\gl)$ contains $q_k(\gl)$ (Lemma \ref{lem:55-B}) and no  $P_j(\gl)$ contains $q_k(\gl)$ with $j<k$, $P_k(\gl)$ must be one of $P_k^+(\gl),$ $P_{k+1}^+(\gl),$ $P_{k+2}^+(\gl),\dots$.  But the $P_k(\gl)$ are pairwise distinct so the only possibility is that $P_k(\gl)=P_k^+(\gl)$.
\end{proof}

\noindent \emph{Proof of Proposition \ref{prop:top-cell}.} 
By Lemma \ref{lem:55-D}, the Goldie rank polynomials of the annihilators of $X_{(1\,c')}$ and $X_{(+\,c')}$ are the same, so $LTC(X_{(1\,c')})$ and $LTC(X_{(+\hspace{1pt}c')})$ are multiples of each other.  We need to show that $C=1$.  By Lemma \ref{lem:55-B}, $T_{(+\,c')}$ occurs in $X_{(1\,c')}$ with multiplicity $1$.  Thus, $1=Cm_{c'}=C$.  Note that $m_{c'}$ is the multiplicity  of $T_{c'}$ in $CC(X_{c'})$, so is $1$.

 The remaining terms in (\ref{eqn:harder}), those involving $T_{(1\,c_1')}$, occur by Proposition \ref{prop:slices-clan}(2).\hfill{$\square$}\\


\subsection{}  We now summarize what has been proved and point out that all multiplicities are one.

Observe that Propositions \ref{prop:cc+}, \ref{prop:cc1}, and  \ref{prop:top-cell} give all characteristic cycles in terms of those of the smaller pair $(G',K')$.  Since all multiplicities are one for $(Sp(2),GL(1))$, induction tells us that multiplicities are one for $(Sp(2n),GL(n))$ for any $n$.

\begin{Thm}\label{thm:main}
Suppose $c\in\Cl$ for $(G,K)=(Sp(2n),GL(n))$.  For $c'\in\Cl'$ write $CC(X_{c'})=\sum [T_{c_j'}]$.  There are three cases.
\begin{enumerate}
\item $c=(+\,c')\in\Chc(k), k=0,1,\dots,n$.  Then $CC(X_c)=\sum [T_{(+\,c_j')}]$.
\item $c=(1\,c')\in\Chc(k), k=0,1,\dots,n-1$.  Then $CC(X_c)=\sum [T_{(1\,c_j')}]$.
\item $c=(1\,c')\in\Chc(n), n$ even.  Then $CC(X_c)=\sum [T_{(1\,c_j')}]+\sum [T_{(+\,c_j')}]$.
\end{enumerate}
\end{Thm}





\appendix 
\section{}\label{app:1} This appendix gives an explicit description, in terms of flags, of the the Schubert varieties appearing in this paper. This description, given in Proposition \ref{prop:sb} below, is used in our computation of normal slices in \S\ref{sec:lemma}.  It is also used to conclude an important property (Cor.~\ref{cor:order}) of the Bruhat order for the relevant Schubert varieties.

As in the body of the paper, $G=Sp(2n)$ and $K=GL(n)$, with the symplectic group realized as in \S\ref{subsec:sp}.
The Cartan subalgebra $\fh$ is the diagonal subalgebra and $\gD^+\subset \gD(\fh,\fg)$ is the positive system with $\fb=\fh+\sum_{\ga\in\gD^+}\fg^{(\ga)}$ the subalgebra of upper triangular matrices.

Recall (\ref{def:our-w}) that
\begin{equation}\label{eqn:our-ww}
\cW:=\{w\in W : -w\rho\text{ is $\gD_c^+$-dominant}\}.
\end{equation}
The Weyl group $W$ may be viewed as permutation of $n$ elements along with sign changes (as in \ref{eqn:w-short}).  If $w=(w_1,\dots,w_n)$, then  $w$  may also be written in the `long form' as elements of $S_{2n}$ as follows.  \begin{equation}
\begin{split}\label{eqn:w-long}
w&=(u_1,u_2,\dots ,u_{2n}), \\
  & u_j=\begin{cases} w_j, &\text{if } w_j>0  \\
              2n+w_j+1,  &\text{if }w_j<0,
\end{cases} \\
&\text{ and }
  u_j+u_{2n-j+1}=n+1, \text{ for }j=1,\dots,n.
\end{split}
\end{equation}

The flag variety $\fB$ of $G$ may be identified with the variety of flags $(F_i)_{i=0,1,\dots,n}$ satisfying 
\begin{equation}\label{eqn:flag-short}
\{0\}=F_0\subset F_i\subset\dots\subset F_n\subset \C^{2n}, \dim(F_i)=i\text{ and $F_i$ is isotropic}
\end{equation}
for all $i=1,\dots,n$.  Equivalently, $\fB$ may be identified with flags 
$(F_i)_{i=0,1,\dots,2n}$ satisfying 
\begin{equation}\label{eqn:flag-long}
\begin{split}
&\{0\}=F_0\subset F_i\subset\dots\subset F_{2n-1}\subset F_{2n}= \C^{2n},   \\
&\dim(F_i)=i,\text{ $F_i$ is isotropic  and } 
F_{2n-i+1}=F_i^\perp i=1,\dots,n.
\end{split}
\end{equation}

Let $B$ be the Borel subgroup with Lie algebra $\fb$. The following description of $B$-orbits in the flag variety is well-known.  Let $x_0:=(V_i), V_i:=\spa\{e_1,\dots, e_i\}$.  Then $B=\Stab(x_0)$.  For each $w\in W$, let $x_w:=w\cdot x_0$.
\begin{Prop}  Let $w\in W$, written in the long form.  Then
$$B_w:=B\cdot x_w=\{(F_i)\in \fB : \dim(F_i\cap V_j)=b_{ij}, i,j=1,2,\dots,n\},$$
where 
\begin{equation}\label{eqn:b}
b_{ij}=^\#\!\!\{l:l\leq i\text{ and }w_l\leq j\}.
\end{equation}
\end{Prop}

When $w\in \cW$ the $B$-orbit $B_w$ has a particularly simple form. To each $w\in \Cal W$ we associate $\tilde{a}=(a_1,\dots,a_n)$ by 
\begin{equation}\label{eqn:a}
a_i=^\#\!\!\{k:k\leq i\text{ and } w_k>0\}.
\end{equation}
When $w$ is written in the long form, this is
\begin{equation}\label{eqn:aa}
a_i=^\#\!\!\{k:k\leq i\text{ and } w_k\leq n\}.
\end{equation}

\begin{Prop}\label{prop:sb}  Let $w\in \cW$ and $\tilde{a}$ as in \ref{eqn:a}.  Then 
$$
Z_w=\bar{B}_w=\{(F_i)\in\fB : \dim(F_i\cap V_n)\geq a_i, i=1,\dots,n\}.
$$
\end{Prop}
Before proving the proposition, we give a lemma  and an example.

If $y,w\in W$ are written in the long form, then the Bruhat order is determined by the following (e.g. \cite[page 30]{BilleyLakshmibai00}).
 \begin{equation} 
  \text{  \parbox{5.0in}{
    \noindent  $y\leq w$ if and only if for each $k=1,\dots,n, \{y_1,\dots,y_k\}\leq \{w_1,\dots,w_k\} $ in the sense that after listing elements of the two sets in increasing order, each element of the first set is less than or equal to the corresponding element of the second  set. }}
\end{equation}

\begin{Lem} Let $b_{ij}^y$ and $b_{ij}^w$ be the integers of (\ref{eqn:b}) for $y$ and $w$, respectively. Then the following hold.
\begin{enumerate}
 \item $y\leq w$ if and only if $b_{ij}^y\geq b_{ij}^w$, $i,j=1,\dots,n$.
 \item $ b_{ij}^w\leq  b_{i,j+1}^w \leq  b_{ij}^w+1$, all $i,j=1,\dots,n$.
 \item If $w\in \cW$, then $b_{in}^w=a_i$.
\end{enumerate}
\end{Lem}
\begin{proof} (1) If $b_{ij}^y\geq b_{ij}^w$ for all $j$, then $\{y_1,\dots,y_k\}$ has at least as many elements less than $j$ as does $\{w_1,\dots,w_k\}$, for all $j$.  So $y\leq w$.  The converse is essentially the same argument.  \newline
\noindent (2) and (3) are immediate.
\end{proof}

It might be useful to consider the following example.    Let $w=(\text{-}143\text{-}2)$; in the long form $w=(8437|2651)$. Then $\tilde{a}=(0,1,2,2)$.  The following table gives the values of $b_{ij}^w$.
\hspace{2cm}
    
\begin{tabular}{|c|cccc|cccc|}
\hline
 $_i\backslash ^j$   & 1 & 2 & 3 & 4 & 5 & 6& 7& 8\\
\hline
1  & 0 & 0 & 0 & 0 & 0 & 0 &0  & 1 \\
2  & 0 & 0 & 0 & 1 & 1 & 1 & 1 & 2     \\
3  & 0 & 0 & 1 & 2 & 2  & 2 & 2 & 3     \\
4  & 0 & 0 & 1 & 2 &  2 & 2 & 3 & 4 \\
\hline
5  & 0 & 1 & 2 & 3 & 3 &3  & 4 &  5    \\
6  & 0 & 1 & 2 & 3 & 3  & 4 & 5 & 6    \\
7  & 0 & 1 & 2 & 3 & 4  & 5 & 6 & 7    \\
8  & 1 & 2 & 3 & 4 &  5 & 6 & 7 & 8  \\
\hline
\end{tabular}\bigskip\newline
Observe that $\tilde{a}$ appears as the fourth column as required by part (3) of the lemma.  Part (a) of the lemma tells us that the upper left block determines the Bruhat order.  Write 
\begin{equation*}
F(\tilde{a}):=\{(F_i)\in\fB : \dim(F_i\cap V_n)\geq a_i, i=1,\dots,n\}
\end{equation*}
For $(F_i)\in F(\tilde{a})$, part (2) of the lemma (applied to $y$) tells us that $b_{ij}^y\geq b_{ij}^w$, $i,j=1,\dots,4$.  It is important that for each $i$, as $j$ runs from $4$ down to $1$, $b_{ij}^w$ decreases by exactly $1$, until it reaches $0$.  Therefore, if $B_y\subseteq F(\tilde{a})$, then by part (1), $y\leq w$.  We conclude that if $B_y\subseteq F(\tilde{A})$, then $B_y\subseteq Z_w$.  Conversely, if $B_y\subset Z_w$ , then $b_{in}^y\geq b_{in}^w=a_i$, so $B_y\subset F(\tilde{a})$.  Therefore, 
$$ F(\tilde{a})=\bigcup_{y\leq w} B_y=Z_w.
$$

\noindent\emph{Proof of the proposition.}  Let $w\in \cW$, so the entries of $-w\rho$ are decreasing (left to right).  This translates, for $w$ as in (\ref{eqn:w-long}), to those entries of $w$ between $1$ and $n$ decreasing (left to right) and the same for those entries between $n+1$ and $2n$.  Let $\tilde{a}^w$ be defined as in (\ref{eqn:a}) for $w$.  Then several facts hold.
\newline\noindent (1) For each $i=1,\dots,n$,
\begin{equation}\label{eqn:fact-1}
b_{ij}^w=\begin{cases} a_i-(n-j), &j=n-a_i,\dots, n-1,n  \\
                       0, & \text{otherwise}.
                       \end{cases}
\end{equation}                      
This follows from the fact that the entries of $w$ between $1$ and $n$ decrease.
\newline\noindent (2) If $b_{ij}^y\geq b_{ij}^w$, all $i,j=i,\dots,n$, then $B_y\subset F(\tilde{a}^w).$  This is clear since the condition implies $b_{in}^y\geq b_{in}^w=a_i^w$.
\newline\noindent (3) $B_y\subset F(\tilde{a}^w)$ implies $b_{ij}^y\geq b_{ij}^w$, $i,j=1,\dots,n$.  For this we assume that $b_{in}^y\geq a_i^w, i=1,\dots,n$.  Fact (1) and part (2) of the lemma implies $b_{ij}^y\geq b_{ij}^w$, $i,j=1,\dots,n$.

We may conclude from these facts that 
$$ F(\tilde{a}^w)=\bigcup_{y\leq w} B_y=Z_w.
$$

\vspace{-35pt}
\hfill{$\square$}

\vspace{30pt}
Recall that (in general) if $\ga$ is a simple root, then $w(\ga)>0$ if and only if $\ell(ws_\ga)=\ell(w)+1.$  Write $\ga_1,\dots,\ga_n$ for the simple roots numbered as in (\ref{eqn:simple}), and $s_1,\dots,s_n$ for the corresponding simple reflections $s_{\ga_1},\dots,s_{\ga_n}$.  For $w\in\cW$, $w(\ga_k)>0$ if and only if $w_k\leq n<w_{k+1}$, when $k<n$, and $w_n<n$, when $k=n$.  This is the case if and only if
$$ b_{ij}^{ws_k}=\begin{cases}
                     b_{ij}^w-1, & (i,j)=(k,w_k)  \\
                     b_{ij}^w,  &\text{otherwise}.
                 \end{cases}    
$$
Also, it follows from part (1) of the lemma and (\ref{eqn:fact-1}) that if $y,w\in\cW$, then 
\begin{equation}\label{eqn:cl-a}
  y\leq w \text{ if and only if } a_i^y\geq a_i^w, i=1,\dots,n.
\end{equation}
Our description of the Schubert varieties $Z_w$, $w\in\cW$ gives the following corollary.
\begin{Cor}\label{cor:order}
If $y,w\in\cW$ and $y\leq w$, then there exist simple roots $\ga_{i_1}\dots,\ga_{i_m}$ so that $ys_{i_1}\cdots s_{i_m}=w$ with $\ell(ys_{i_1}\dots s_{i_k})=\ell(y)+k$ and $ys_{i_1}\cdots s_{i_k}\in\cW$, for all $k=1,\dots,m$.
\end{Cor}

In other words, the closure relations among Schubert varieties $Z_w$, $w\in\cW$, are `generated by' \emph{simple} reflections.  This corollary follows from similar statement in \cite[\S1]{LakshmibaiWeyman90}.

\begin{proof}
We prove the statement by induction on $m=\ell(w)-\ell(y)$.  When $m=1$ the statement is immediate.  Assume $m>1$ and let $l$ be the smallest index for which $a_l^y>a_l^w$.  Thus, $a_k^y=a_k^w$, $k=1,\dots,l-1$, so by (\ref{eqn:fact-1}) the first $l-1$ entries of $y$ and $w$ coincide.  Let $w_l=p, y_l=q_1$; therefore $q_1\leq n<p$ and 
\begin{align*}
  &w=(\cdots \;p\;\;\cdots\;\; \;\;\cdots\;\;| \;\;\cdots \;\;)  \\
  &w=(\cdots\; q_1\;\cdot \,\,q_rp\;\cdots \;| \;\;\cdots \;\;), \;\; q_i\leq n,  \\
\intertext{or}
  &w=(\cdots \;p\;\;\cdots\;\cdots\;| \;\;\cdots \;\;)  \\
  &w=(\cdots\; q_1\;\;\cdots \;\;q_r\;| \;\;\cdots \;\;),\;\; q_i\leq n.  \\
\end{align*}
Suppose that $q_r=y_k$.  Then in both cases $ys_k\geq y$ and $ys_k\in\cW$.  Therefore, induction applies giving  $i_2,\dots,i_m$ with $ys_1s_{i_2}\dots s_{i_m}$ as in the statement of the corollary.
\end{proof}


\section{} \label{app:B}  
It is well-known that the associated variety of a highest weight Harish-Chandra module is contained in $\fp_+$ (\cite{NishiyamaOchiaiTaniguchi01}).  In this appendix we we show the converse.  It is likely that this fact is known, but we include a proof since we could not find one in the literature.  

Let $X$ be  a finitely generated Harish-Chandra module. The associated variety of $X$ may be computed as follows (\cite{Vogan91}).  Choose a finite dimensional $K$-stable generating subspace of $X$ and filter $X$ by $X_n:=\Cal U_n(\fg)X_0$.  Here $\Cal U_n(\fg)$ is the usual filtration of $\Cal U(\fg)$ by degree.  Let $I_X$ be the annihilator of $\gr(X)=\oplus X_n/X_{n-1}$ in $P(\fg^*)$ ($\simeq S(\fg)\simeq \gr(\Cal U_n(\fg))$).  Then the associated variety of $X$ is the affine variety in $\fg^*$ on which all polynomials in $I_X$ vanish.  Since the filtration is $K$-stable,  $AV(X)\subset (\fg/\fk)^*$.  Under the identification $\fg^*\simeq\fg$ (via the Killing form) we have $(\fg/\fk)^*\simeq\fp$ and $P(\fg^*)\simeq P(\fg)$, and we identify $I_X$ with an ideal in $P(\fg)$.

Now assume $G_\R$ is a group of hermitian type and $X$ is any irreducible Harish-Chandra module.  Suppose $X$ has a highest weight vector (with respect to a positive system containing $\gD(\fp_+)$).  Take $X_0$ to be the $K$-type generated by the highest weight vector, then $X_n=\Cal U_n(\fp_+)X_0$, so $\fp_+\subset I_X$.  It follows that $AV(X)\subset\fp_+$.  We show the converse.

\begin{Prop}\label{prop:converse} Suppose $G_\R$ is of hermitian type  and $X$ is any irreducible Harish-Chandra module for which $AV(X)\subset \fp_+$.  Then $X$ is a highest weight module.
\end{Prop}

\begin{proof}  Let $\{x_1,\dots, x_d\}$ be a basis of $\fp_+$.  Then there is an $m'\in\N$ so that $x_i^{m'}\in I_X$, for all $i$, since $\fp_+\subset\{p\in P(\fg) : p \text{ vanishes on }AV(X)\}=\sqrt{I_X}$.  Now take $m=m'd$ and we see that any product $x_1\dots x_m\in I_X$ with $x_i\in\fp_+$.  Therefore,
\begin{equation}\label{eqn:a-2-1} \Cal U_m(\fp_+):=\Cal U_m(\fg)\cap\Cal U(\fp_+)\subset I_X.
\end{equation}
Set $Y_n:=\Cal U_n(\fp_+)X_0$.  Then we have 
$$  \fp_+Y_{m-1}=\Cal U_m(\fp_+)X_0\subset X_{m-1},
$$
(by the definition of $\gr(X_n)$).  Therefore,
$$\fp_+Y_{m-1}\subset X_{m-1}\cap Y_m=Y_{m-1}.
$$
Thus, $Y_{m-1}$ is finite dimensional and $\fk+\fp_+$ stable. So if $Y_{m-1}\neq 0$, then $Y_{m-1}$ contains a highest weight vector.  If $Y_{m-1}=0$, then choose the largest $k\leq m-1$ for which $Y_k\neq 0$.  Then $Y_k$ has a highest weight vector.
\end{proof}

\providecommand{\bysame}{\leavevmode\hbox to3em{\hrulefill}\thinspace}
\providecommand{\MR}{\relax\ifhmode\unskip\space\fi MR }
\providecommand{\MRhref}[2]{%
  \href{http://www.ams.org/mathscinet-getitem?mr=#1}{#2}
}
\providecommand{\href}[2]{#2}

\end{document}